\documentclass[reqno]{amsart}

\usepackage{amsmath,amsthm,amsfonts, amssymb,amscd}
\usepackage[all]{xy}
\numberwithin{equation}{section}

\newtheorem{theorem}{Theorem}[section]
\newtheorem{lemma}[theorem]{Lemma}
\newtheorem{definition}[theorem]{Definition}
\newtheorem{proposition}[theorem]{Proposition}

\newcommand{\dx}{\, \mbox{\rm d}}

\newcommand{\Dim}{\mbox{\rm Dim}}
\newcommand{\Def}{\mbox{\rm Def}}

\usepackage{enumerate}

\begin{document}

\title[Lifting, $n$-Dimensional Spectral Resolutions, and $n$-Dimensional Observables]{Lifting, $n$-Dimensional Spectral Resolutions, and $n$-Dimensional Observables}
\author[A. Dvure\v{c}enskij, D. Lachman]{Anatolij Dvure\v{c}enskij$^{1,2}$, Dominik Lachman$^2$}
\maketitle

\begin{center}  \footnote{


The paper has been supported by the grant of
the Slovak Research and Development Agency under contract APVV-16-0073 and the grant VEGA No. 2/0142/20 SAV, AD, and by grant CZ.02.2.69/0.0/0.0/16-027/0008482 SPP 8197200115, D.L.}
Mathematical Institute,  Slovak Academy of Sciences\\
\v Stef\'anikova 49, SK-814 73 Bratislava, Slovakia\\
$^2$ Depart. Algebra  Geom.,  Palack\'{y} Univer.\\
17. listopadu 12, CZ-771 46 Olomouc, Czech Republic\\

E-mail: {\tt
dvurecen@mat.savba.sk,\quad dominiklachman@seznam.cz}
\end{center}

\begin{abstract}
We show that under some natural conditions, we are able to lift an $n$-dimensional spectral resolution from one monotone $\sigma$-complete unital po-group into another one, when the first one is a $\sigma$-homomorphic image of the second one. We note that an $n$-dimensional spectral resolution is a mapping from $\mathbb R^n$ into a quantum structure which is monotone, left-continuous with non-negative increments and which is going to $0$ if one variable goes to $-\infty$ and it goes to $1$ if all variables go to $+\infty$. Applying this result to some important classes of effect algebras including also MV-algebras, we show that there is a one-to-one correspondence between $n$-dimensional spectral resolutions and $n$-dimensional observables on these effect algebras which are a kind of $\sigma$-homomorphisms from the Borel $\sigma$-algebra of $\mathbb R^n$ into the quantum structure. An important used tool are two forms of the Loomis--Sikorski theorem which use two kinds of tribes of fuzzy sets. In addition, we show that we can define three different kinds of $n$-dimensional joint observables of $n$ one-dimensional observables.
\end{abstract}

\section{Introduction}

Whenever we have a surjective homomorphism $\pi: A \to B$, where $A$ and $B$ are mathematical structures, we can ask whether there is a right inverse of $\pi$, that is a homomorphism $\psi: B \to A$ such that $\pi\circ \psi$ is the identity on $B$. The mapping $\psi$ is also said to be a lifting. In general, the answer is negative, but in some special cases, we can find a positive answer.  The first important result in this direction is the von Neumann-Maharam lifting theorem concerning probability spaces, see \cite{Mah}, \cite[Thm 341K]{Fre}.

Motivated by the notion of a lifting,  the aim of the present paper is two-fold:

(1) Let $(G,u)$ and $(H,v)$ be Abelian unital Dedekind $\sigma$-complete unital $\ell$-groups (or Dedekind monotone $\sigma$-complete unital po-groups with interpolation) and let $\pi:(G,u)\to (H,v)$ be a surjective $\sigma$-homomorphism. Given an $n$-dimensional spectral resolution $F:\mathbb R^n\to \Gamma(H,v)$, which is a left-continuous, monotone mapping with non-negative increments that is going to $0$ if one coordinate goes to $-\infty$ and to $1$ if all coordinates go to $+\infty$, find an $n$-dimensional spectral resolution $K:\mathbb R^n \to \Gamma(G,u)$ such that $\pi\circ K=F$.

(2) Applying (1), show that for some important cases of quantum structures there is a one-to-one correspondence between $n$-dimensional spectral resolutions and $n$-dimensional observables.

The reason for such a study is rasing up from the study of observables and spectral resolutions on quantum structures. We recall that quantum structures are nowadays different algebraic structures (orthomodular lattices, orthomodular posets, orthoalgebras, effect algebras, MV-effect algebras, etc., see \cite{DvPu}) which model quantum mechanical events. As it is well-known, Kolmogorov probability models are not adequate structures for describing measurement in quantum mechanics, see \cite{BiNe,Var}. In the last three decades, effect algebras \cite{FoBe} together with MV-effect algebras became the most important algebras of quantum structures. Their orthodox example is the system $\mathcal E(H)$ of Hermitian operators of a real, complex or quaternionic Hilbert space $H$ that are between the zero and identity operator. It is the interval in the partially ordered group $\mathcal B(H)$ of all Hermitian operators. Another quantum structure connected with a Hilbert space quantum mechanics is the system $\mathcal P(H)$ of all orthogonal projectors in $H$; we note that $\mathcal P(H)$ is a complete orthomodular lattice.

Every measurement in classical physics is modeled within frames of a Kolmogorov model $(\Omega,\mathcal S,P)$, where $P$ is the probability measure on the $\sigma$-algebra $\mathcal S$ of subsets of a set $ \Omega \ne \emptyset$ of elementary events. Then we measure by a random variable $f:\Omega \to \mathbb R$ which is $\mathcal S$-measurable, i.e. the mapping $x_f:\mathcal B(\mathbb R)\to \mathcal S$ given by $x_f(A)=f^{-1}(A)$, $A\in \mathcal B(\mathbb R)$, is a kind of a $\sigma$-homomorphism; we call this homomorphism an observable. If we measure in non-classical physics, then by an observable is meant a kind of a $\sigma$-homomorphism from the Borel $\sigma$-algebra $\mathcal B(\mathbb R)$ into a quantum structure. For example, observables in $\mathcal E(H)$ are so-called positive operator-valued measures, and ones in $\mathcal P(H)$ are projector-valued measures, see e.g. \cite{DvPu}.

If we measure simultaneously more variables, then the Heisenberg uncertainty principle appearing in quantum mechanics does not allow in many important cases to use a Kolmogorov model. For measurement of $n$ quantities, we use $n$-dimensional observables which are also a kind of $\sigma$-homomorphisms defined on the Borel $\sigma$-algebra $\mathcal B(\mathbb R^n)$ with values in the quantum structure. If $x$ is an $n$-dimensional observable, then the mapping $F(s_1,\ldots,s_n):=x((-\infty,s_1)\times \cdots\times (-\infty,s_n))$, $s_1,\ldots,s_n \in \mathbb R$, is an $n$-dimensional spectral resolution which can be characterized as a mapping on $\mathbb R^n$ with values in the algebra of events which is monotone, left-continuous with non-negative increments going to $0$ if one variable goes to $-\infty$ and going to $1$ if all variables go to $+\infty$. Therefore, there is an important question when does $n$-dimensional spectral resolution imply the existence of the corresponding $n$-dimensional observable?

A positive answer to this question in the case of one-dimensional spectral resolutions was given in  \cite{Cat} for quantum  logics and in \cite{DvKu} for the case of $\sigma$-complete MV-algebras as well as for monotone $\sigma$-complete effect algebras with the Riesz Decomposition Property (RDP). For two-dimensional spectral resolutions this problem was solved in \cite{DvLa1}. For more-dimensional ones we have partial results for special $\sigma$-complete  MV-algebras in \cite{DvLa2} using measure-theoretical reasoning from \cite{RiNe,Fre}. A general solution was not possible to present using methods from \cite{DvLa1,DvLa2}, and therefore, we developed methods using lifting of the $n$-dimensional spectral resolutions.

The paper is organized as follows. The necessary information about quantum  structures is gathered in Section 2 and $n$-dimensional spectral resolutions and their connections to $n$-dimensional observables are explained in Section 3. Basic results, lifting of spectral resolutions when the surjective homomorphism satisfies the lifting property, are given in Section 4. In Section 5, we establish a one-to-one correspondence between $n$-dimensional spectral resolutions and $n$-dimensional observables on $\sigma$-complete MV-effect algebras and monotone $\sigma$-complete effect algebras with (RDP). To show that, we use the established lifting. In addition, we develop also another approach using Loomis--Sikorski theorems for our algebras. Finally, in Section 6, we apply results of Section 5 to introduce two different kinds of $n$-dimensional joint observables.

\section{Basic Notions and Results}

In the section, we gather the necessary definitions and results which will be used in the text.

In the nineties, Foulis and Bennett introduced in \cite{FoBe} a new quantum structure, called an {\it effect algebra} which is a partial algebra $E=(E;+,0,1)$ with a partially defined operation $+$ that is an associative and commutative operation such that

\begin{enumerate}

\item[(i)] for any $a \in E$, there exists a unique element $a'
\in E$ such that $a+a'=1$;

\item[(ii)] if $a+1$ is defined in $E$, then $a=0$.
\end{enumerate}
If we have $a+c=b$, then we put $a\le b$ and $\le$ is a partial ordering on $E$ such that $0$ and $1$ are the least and last elements of $E$. For the unique element $c\in E$ such that $a+c=b$, we write $c = b-a$.

Very important examples of effect algebras are originated in Abelian po-groups. We remind that an Abelian group $(G;+,-,0)$ with a fixed partial order $\le$ is a {\it partially ordered group} (po-group in abbreviation), if $g\le h$ for $g,h \in G$, implies $g+k\le h+k$ for each $k \in G$. If the order $\le$ is a lattice one, we say that $G$ is a lattice-ordered group ($\ell$-group for short). We note that all po-groups used in the paper are Abelian, so we will not underline that a po-group is Abelian. A positive cone of a po-group $G$ is the set $G^+=\{g \in G\colon g\ge 0\}$. An element $u\in G^+$ is a {\it strong unit} of $G$, if given $g\in G$, there is an integer $n\ge 1$ such that $g\le nu$. A couple $(G,u)$, where $G$ is a po-group with a fixed strong unit $u\in G^+$, is said to be a {\it unital po-group}. A po-group $G$ is with {\it interpolation} if $g_1,g_1\le h_1,h_2$ entails an element $k\in G$ such that $g_1,g_2\le k \le h_1,h_2$. For more information about po-groups, we recommend the monograph \cite{Goo}.

Now, let $(G,u)$ be a unital po-group, and let $[0,u]=\{g\in G\colon 0\le g\le u\}$. Then $\Gamma^e_a(G,u)=([0,u];+,0,u)$ is an effect algebra, called an interval one. A sufficient condition to be $E$ an interval effect algebra is the {\it Riesz Decomposition Property} ((RDP), in abbreviation): If $a_1+a_2=b_1+b_2$, there are four elements $\{c_{ij}\in E\colon i,j \in \{1,2\}\}$ such that $a_1=c_{11}+c_{12}$, $a_2=c_{21}+c_{22}$, $b_1=c_{11}+c_{21}$ and $b_2=c_{12}+c_{22}$. The principal representation result, \cite{Rav}, says that given an effect algebra $E$ with (RDP), there is a unique (up to isomorphism of unital po-groups) unital po-group $(G,u)$ such that $E \cong \Gamma^e_a(G,u)$.

The operation $+$ is associative, so that we can write $\sum_{i=1}^n a_i$ for $a_1,\ldots,a_n$ whenever it exists in $E$. An infinite system of elements $\{a_t\}_{t \in T}$ is said to be {\it summable}, if for each finite subset $P$ of $T$, the element $a_P=\sum_{t\in P}a_t$ exists in $E$, and in addition, if the element $a=\bigvee\{a_P\colon P\subseteq T \mbox{ \rm finite }\}$ is defined in $E$, the element $a$ is said to be the {\it sum} of $\{a_t\colon t \in T\}$ and we write $a=\sum_{t \in T}a_t$.

An important subclass of effect algebras consists of MV-effect algebras which are equivalent to MV-algebras. We say that an MV-{\it algebra} is an algebra $M=(M;\oplus,',0,1)$ of type $(2,1,0,0)$, where $(M;\oplus,0)$ is a commutative monoid with the neutral element $0$ and for all $a,b\in M$, we have:
\begin{enumerate}
	\item[(i)] $a''=a$;
	\item[(ii)] $a\oplus 1=1$;
	\item[(iii)] $a\oplus (a\oplus b')'=b\oplus (b\oplus a')'$.
\end{enumerate}
In any MV-algebra $(M;\oplus,',0,1)$, we can also define the following term operation:
\[
a\odot b:=(a'\oplus b')'.
\]
Property (iii) implies that $a\vee b=a\oplus (a\oplus b')'$ is the join of $a$ and $b$, and then $a\le b$ if there is $c\in M$ such that $a\oplus c=a$, or equivalently, if $a\vee b = b$, is a partial order on $M$ and $M$ is a distributive lattice.

If we take a unital $\ell$-group $(G,u)$, and $[0,u]=\{g \in G\colon 0\le g \le u\}$, then $\Gamma(G,u)=([0,u];\oplus,',0,u)$, where $a\oplus b =(a+b)\wedge u$, $a'=u-a$, $a,b \in[0,u]$, is an MV-algebra. Due to \cite{Mun}, for every MV-algebra, there is a unique unital $\ell$-group $(G,u)$ such that $M\cong \Gamma(G,u)$. For more info about MV-algebras, we recommend the book \cite{CDM}.

On every MV-algebra, we can define a partial addition $+$ such that $a+b$ is defined iff $a\odot b = 0$ and then $a+b:= a\oplus b$. Then $(M;+,0,1)$ is an effect algebra (called an {\it MV-effect algebra}) with (RDP) which is a lattice. We note that MV-algebras play a similar role as do Boolean algebras for orthomodular lattices.

We say that a poset $(G;\le)$ is {\it monotone $\sigma$-complete} provided that every ascending (descending) sequence $x_1\le x_2\le \cdots$ ($x_1\ge x_2\ge \cdots$) in $G$ which is bounded above (below) in $G$ has a supremum (infimum) in $G$.  If $G$ is a po-group, then a monotone $\sigma$-complete po-group $G$ is said to be also {\it Dedekind $\sigma$-complete}.

A {\it state} on an effect algebra $E$ is a mapping $s:E \to [0,1]$ such that (i) $s(1)=1$ and (ii) $s(a+b)=s(a)+s(b)$ whenever $a+b$ exists in $E$. We denote by $\mathcal S(E)$ the set of all states on $E$. It can happen that the state space $\mathcal S(E)$ is empty, see e.g. \cite{Gre}. However, if $E\ne \{0\}$ has (RDP), then it possesses at least one state. A state $s$ is {\it extremal} if from $s=\lambda s_1 +(1-\lambda)s_2$, where $s_1,s_2\in \mathcal S(E)$ and $\lambda \in (0,1)$, we conclude $s=s_1=s_2$. We denote by $\mathcal \partial S(E)$ the set of extremal states on $E$. A state $s$ is $\sigma$-additive if from $a=\sum_{n=1}^\infty a_n$, we have $s(a)=\sum_{n=1}^\infty s(a_n)$.

A state on a unital po-group $(G,u)$ is a group homomorphism $s:G \to \mathbb R$ such that (i) $s(g)\ge 0$ whenever $g \in G^+$ and (ii) $s(u)=1$. The restriction of any state on $(G,u)$ onto $E=\Gamma^e_a(G,u)$ is a state on $E$, and conversely, if $E$ has (RDP), then every state on $\Gamma^e_a(G,u)$ can be uniquely extended to a state on $(G,u)$.

We say that a net of states $\{s_\alpha\}_\alpha$ {\it converges weakly} to a state $s$ if $s(a)=\lim_\alpha s_\alpha(a)$ for each $a\in E$. The weak topology is a Hausdorff one. If $E$ is an MV-algebra, then $\mathcal S(E)$ is a compact space, if $E$ is an effect algebra with (RDP), then the state space is a Choquet simplex, \cite[Thm 10.17]{Goo}, not necessarily compact. According to a delicate result of Choquet \cite[p. 49]{Alf}, $\partial\mathcal S(E)$ is always a Baire space, i.e. the Baire Category Theorem holds for $\partial \mathcal S(E)$.

As a basic source of information about effect algebras and for unexplained notions and results, we recommend to consult with \cite{DvPu}.

\section{Basic Properties of $n$-Dimensional Spectral Resolutions and $n$-Dimensional Observables}

In the section, we define the main notions of the paper~-- $n$-dimensional spectral resolutions and $n$-dimensional observables defined on monotone $\sigma$-complete effect algebras and on $\sigma$-complete MV-algebras, and we present the main properties of $n$-dimensional spectral resolutions

Let $(s_1,\ldots,s_n), (t_1,\ldots,t_n)\in \mathbb R^n$ be two $n$-tuples of reals. If we write $(s_1,\ldots,s_n)\le (t_1,\ldots,t_n)$, this means that $s_i\le t_i$ for each $i=1,\ldots,n$. Then $(s_1,\ldots,s_n)< (t_1,\ldots,t_n)$ means that each $s_i\le t_i$ and for some $i=1,\ldots,n$, $s_i<t_i$. The strict ordering $(s_1,\ldots,s_n)\ll (t_1,\ldots,t_n)$ means $s_i<t_i$ for each $i=1,\ldots,n$.

Let $n\ge 1$ be a fixed integer. For each $i=1,\ldots,n$ and for all are real numbers $a_i \le b_i$,
we define an operator $\Delta_i(a_i,b_i)$ which, for any mapping $H: \mathbb R^n\to G$, assigns $\Delta_i(a_i,b_i)H(s_1,\ldots,s_n)=H(s_1,\ldots,s_{i-1},b_i, s_{i+1},\ldots,s_n) - H(s_1,\ldots,s_{i-1},a_i, s_{i+1},\ldots,s_n)$.

An $n$-dimensional spectral resolution was defined in \cite{DvLa2} as follows:

\begin{definition}\label{de:3.1}
{\rm Let $(G,u)$ be a Dedekind monotone $\sigma$-complete po-group and let $n\ge 1$ be an integer. An $n$-{\it dimensional spectral resolution} on $E=\Gamma^e_a(G,u)$ is any mapping $F:\mathbb R^n \to \Gamma^e_a(G,u)$ such that}

\begin{equation}\label{eq:(3.1)}
F(s_1,\ldots,s_n)\le F(t_1,\ldots,t_n) \quad \mbox{\rm if}\quad (s_1,\ldots,s_n)\le (t_1,\ldots,t_n),
\end{equation}
\begin{equation}\label{eq:(3.2)} \bigvee_{(s_1,\ldots,s_n)}F(s_1,\ldots,s_n)=u,
\end{equation}
\begin{equation}\label{eq:(3.3)} \bigvee_{(s_1,\ldots,s_n)\ll(t_1,\ldots,t_n)}F(s_1,\ldots,s_n) = F(t_1,\ldots,t_n),
\end{equation}
\begin{equation}\label{eq:(3.4)}
\bigwedge_{t_i} F(s_1,\ldots, s_{i-1},t_i, s_{i+1},\ldots, s_n)=0 \mbox{ \rm for } i=1,\ldots,n,
\end{equation}

\begin{equation}\label{eq:(3.5)}
\Delta_{1}(a_1,b_1)\cdots \Delta_{n}(a_n,b_n)F(s_1,
\ldots,s_n)\ge 0\ \mbox{ (volume condition)}.
\end{equation}
\end{definition}
It is important to note that monotonicity of $F$ and Dedekind $\sigma$-completeness of $G$ entails that all suprema and infimum on the left-hand sides of (\ref{eq:(3.2)})--(\ref{eq:(3.4)}) exist in $G$. Indeed, to see (\ref{eq:(3.2)}), let $\{(s^l_1,\ldots,s^l_n)\}_l$ and $\{(u^l_1,\ldots,u^l_n)\}_l$ be two non-decreasing sequences in $\mathbb R^n$ going to $(+\infty,\ldots,+\infty)$. The monotonicity of $F$ entails that the following suprema exist in $G$ and
$$
\bigvee_l F(s^l_1,\ldots,s^l_n)=\bigvee_{(s_1,\ldots,s_n)} F(s_1,\ldots,s_n)= \bigvee_l F(u^l_1,\ldots,u^l_n).
$$
Similarly, if $\{(s^l_1,\ldots,s^l_n)\}_l\nearrow (t_1,\ldots,t_n)$ and $\{(u^l_1,\ldots,u^l_n)\}_l\nearrow (t_1,\ldots,t_n)$ are two sequences of reals in $\mathbb R^n$ such that $(s^l_1,\ldots,s^l_n), (t^l_1,\ldots,t^l_n)\ll (t_1,\ldots,t_n)$ for each $l\ge 1$, then the next elements exist in $G$ and
\begin{align*}
\bigvee_l F(s^l_1,\ldots,s^l_n)&=\bigvee_{(s_1,\ldots,s_n)\ll (t_1,\ldots,t_n)} F(s_1,\ldots,s_n)= \bigvee_{(s_1,\ldots,s_n)\le (t_1,\ldots,t_n)} F(s_1,\ldots,s_n)= \bigvee_l F(u^l_1,\ldots,u^l_n)\\
&=F(t_1,\ldots,t_n).
\end{align*}
In a similar way we can establish that also the infimum in (\ref{eq:(3.4)}) exists in $G$.


We note that $n$-dimensional spectral resolutions are strongly connected with the so-called $n$-dimensional observables. We remind that an {\it $n$-dimensional observable} on a monotone $\sigma$-complete effect algebra $E$ is a mapping $x$ defined on the Borel $\sigma$-algebra $\mathcal B(\mathbb R^n)$ with values in $E$ such that (i) $x(\mathbb R^n)=1$, (ii) $x(A\cup B)=x(A)+x(B)$ whenever $A\cap B=\emptyset$, and (iii) $\{A_i\}_i\nearrow A$ implies $\bigvee_i x(A_i)=x(A)$.

If, given an $n$-dimensional observable $x$ on $E=\Gamma^e_a(G,u)$, we define
a function $F_x: \mathbb R^n \to \Gamma^e_a(G,u)$ by
$$
F_x(s_1,\ldots,s_n)=x((-\infty,s_1)\times \cdots \times (-\infty,s_n)),\quad (s_1,\ldots,s_n)\in \mathbb R^n,
$$
then $F_x$ is an $n$-dimensional spectral resolution. We note that the volume condition means that if $A=\langle a_1,b_1)\times\cdots\times \langle a_n,b_n)$, $a_i\le b_i$ for each $i=1,\ldots,n$, denotes an $n$-dimensional semi-closed block, then
$$
\Delta_{1}(a_1,b_1)\cdots\Delta_{n}(a_n,b_n)F_x(s_1,
\ldots,s_n)=x(A)\ge 0.
$$

Now we present an example of an $n$-dimensional observable and consequently of an $n$-dimensional spectral resolution: Let $\{t_k\}_k$ be a finite or countable set of mutually different elements of $\mathbb R^n$ and let $\{a_k\}_k$ be a finite or countable family of summable elements of $\Gamma^e_a(G,u)$ such that $\sum_k a_k =u$, where $G$ is a Dedekind monotone $\sigma$-complete po-group. Then
$$
x(A)=\sum_{k\colon t_k\in A} a_k,\quad A\in \mathcal B(\mathbb R^n),
$$
is an $n$-dimensional observable and $F_x$ is an example of an $n$-dimensional spectral resolution.

Our main task is to show the converse statement, i.e. given an $n$-dimensional spectral resolution $F$, find a unique $n$-dimensional observable $x$ such that  $F(s_1,\ldots,s_n)=x((-\infty,s_1)\times \cdots \times (-\infty,s_n))$ holds for each $(s_1,\ldots,s_n)\in \mathbb R^n$.

Now, we exhibit the basic properties of $n$-dimensional spectral resolutions which were established in \cite[Prop 3.3]{DvLa2}.

\begin{proposition}\label{pr:prop}
Let $F$ be an $n$-dimensional spectral resolution on $\Gamma^e_a(G,u)$, where $(G,u)$ is a unital Dedekind monotone $\sigma$-complete po-group.

{\rm (1)} If $(i_1,\ldots,i_n)$ is any permutation of $(1,\ldots,n)$, then
\begin{equation}\label{eq:(3.6)}
\Delta_{1}(a_1,b_1)\cdots \Delta_{n}(a_n,b_n)F(s_1,
\ldots,s_n) = \Delta_{i_1}(a_{i_1},b_{i_1})\cdots \Delta_{i_n}(a_{i_n},b_{i_n})F(s_1,\ldots,s_n).
\end{equation}

{\rm (2)} If $i_1,\ldots,i_k$ are mutually different integers from $\{1,\ldots,n\}$ for $1\le k< n$, then
\begin{equation}\label{eq:(3.7)}
0\le \Delta_{i_1}(a_{i_1},b_{i_1})\cdots \Delta_{i_k}(a_{i_k},b_{i_k})F(s_1,\ldots,s_n)\le u.
\end{equation}

Given a semi-closed block $A=\langle a_1,b_1)\times \cdots\times \langle a_n,b_n)$, we define $V_F(A)$ to be the left-hand side of the volume condition {\rm (\ref{eq:(3.5)})}, then $V_F(A)\in \Gamma^e_a(G,u)$.
\end{proposition}

\begin{proof}
First we note that if $\{a_i\}_i\searrow a$ and $\{b_i\}_i\searrow b$, then \begin{equation}\label{eq:mono}
\{a_i+b_i\}_i\searrow a+b.
\end{equation}
We outline only the proof of (2).
Let $j_1,\ldots,j_{n-k}$ be those indices from $\{1,\ldots,n\}$ which are different of $i_1,\ldots,i_k$. Let us expands the volume condition (\ref{eq:(3.5)}) and let express it in the form $L\ge R$, where on both sides are now only non-negative terms. First we are going with $a_{i_l}\searrow -\infty$ in $R$, so that  $\bigwedge_{a_{i_l}}v=0$ by (\ref{eq:(3.4)}) for $l=1,\ldots,n-k$ and apply (\ref{eq:mono}) for each term $v$ in $R$ containing $a_{i_l}$; we obtain $\hat R$ and $L\ge \hat R$. Then we do the same with $L$ and we obtain $\hat L$. Hence, $\hat L\ge \hat R$, so that
$$
\Delta_{i_1}(a_{i_1},b_{i_1})\cdots \Delta_{i_k}(a_{i_k},b_{i_k})F(s_1,\ldots,s_n)\ge 0.
$$
Consequently,

\begin{align*}
&\Delta_{i_{k+1}}(a_{i_{k+1}},b_{i_{k+1}})
\Delta_{i_1}(a_{i_1},b_{i_1})\cdots \Delta_{i_k}(a_{i_k},b_{i_k})F(s_1,\ldots,s_n)=\\
&\Delta_{i_1}(a_{i_1},b_{i_1})\cdots \Delta_{i_k}(a_{i_k},b_{i_k})F(s_1,\ldots,b_{i_{k+1}},\ldots, s_n)-
\Delta_{i_k}(a_{i_k},b_{i_k})F(s_1,\ldots,a_{i_{k+1}},\ldots, s_n)\ge 0,
\end{align*}
that is
\begin{equation}\label{eq:(3.8)}
\Delta_{i_k}(a_{i_k},b_{i_k})F(s_1,\ldots,b_{i_{k+1}},\ldots, s_n)\ge
\Delta_{i_k}(a_{i_k},b_{i_k})F(s_1,\ldots,a_{i_{k+1}},\ldots, s_n).
\end{equation}
By induction with respect to $k$, we establish (2).
\end{proof}

Every operator $\Delta_i(a,b)$ can be defined also for $a=-\infty$ and $b \in \mathbb R$ as $\Delta_i(a,b)G(s_1,\ldots,s_n)= G(s_1,\ldots,s_{i-1},b,s_{i+1},\ldots,s_n)$ if $G(s_1,\ldots,s_n)\in \Gamma^e_a(G,u)$ and (\ref{eq:(3.4)}) holds for $G$. If $A=C_1\times\cdots\times C_n$, where either $C_i= \langle a_i,b_i)$ for $-\infty <a_i\le b_i <\infty$, or $C_i =(-\infty,b_i)$ for $b_i\in \mathbb R$, we can define $V_F(A)= \Delta_1(a_1,b_1)\cdots\Delta_n(a_n,b_n)F(s_1,\ldots,s_n)$. Then $0\le V_F(A)\le u$. In particular, if each $C_i=(-\infty,b_i)$, then $V_F(A)=F(b_1,\ldots,b_n)$.

Every $n$-dimensional semi-closed block $A=\langle a_1,b_1)\times \cdots \times \langle a_n,b_n)$ has $2^n$ vertices $\alpha = (\alpha_1,\ldots,\alpha_n)$, where $\alpha_i \in \{a_i,b_i\}$ for each $i=1,\ldots, n$. For each vertex $\alpha=(\alpha_1,\ldots,\alpha_n)$, we set $|\alpha|$ as the number of $\alpha_i$'s coinciding with $a_i$ in $\alpha=(\alpha_1,\ldots,\alpha_n)$.
Then the volume condition can be expressed also in the form
\begin{equation}\label{eq:volum}
\Delta_{1}(a_1,b_1)\cdots \Delta_{n}(a_n,b_n)F(s_1,
\ldots,s_n)=\sum_\alpha (-1)^{|\alpha|}F(\alpha_1,\ldots,\alpha_n).
\end{equation}

Equation (\ref{eq:(3.7)}) has the following important interpretation. Let $A=\langle a_1,b_1)\times \cdots \times \langle a_n,b_n)$ be an $n$-dimensional semi-closed block. Fix mutually different integers $i_1,\ldots,i_k$ from $\{1,\ldots,n\}$ and choose $c_i\in \{a_i,b_i\}$ if $i \in \{1,\ldots,n\}\setminus \{i_1,\ldots,i_k\}$. For each $i \in \{1,\ldots,n\}$, we set $C_i=\langle a_i,b_i)$ if $i\in \{i_1,\ldots,i_k\}$ and $C_i=\{c_i\}$ otherwise. If we define a semi-closed block $A_k= C_1\times \cdots\times C_n$, then $\dim A_k=k$. Hence, if in (\ref{eq:(3.7)}) we put $s_i=c_{i}$ for $i \in \{1,\ldots,n\}\setminus \{i_1,\ldots,i_k\}$, then we have a kind of the volume condition for the subblock $A_k$ whose vertices coincide with some vertices of $A$; we can call it an $(n-k)$-th derived volume condition. If $k=n$, we have the original volume condition (\ref{eq:(3.5)}). On the other hand, if $k=0$, then we obtain only a vertex $\alpha =(c_1,\ldots,c_n)$, and for it we have also a special kind of the volume condition $F(c_1,\ldots,c_n)$ which lies of course between $0$ and $u$.

On the other hand, if for $i_1,\ldots,i_k$ and fixed $s_{j_1},\ldots,s_{j_{n-k}}$, where $\{j_1,\ldots,j_{n-k}\}=\{1,\ldots,n\}\setminus \{i_1,\ldots,i_k\}$, we define $F_{s_{j_1},\ldots,s_{j_{n-k}}}:\mathbb R^k\to \Gamma^e_a(G,u)$ by  $F_{s_{j_1},\ldots,s_{j_{n-k}}}(s_{i_1},\ldots,s_{i_k}):=F(s_1,\ldots,s_n)$, then every $F_{s_{j_1},\ldots,s_{j_{n-k}}}$ satisfies properties (\ref{eq:(3.3)})--(\ref{eq:(3.5)}) of a $k$-dimensional spectral resolution, and
$$
\Delta_{i_1}(a_{i_1},b_{i_1})\cdots \Delta_{i_k}(a_{i_k},b_{i_k})F_{s_{j_1},\ldots,s_{j_{n-k}}} (s_{i_1},\ldots,s_{i_k})
=\Delta_{i_1}(a_{i_1},b_{i_1})\cdots \Delta_{i_k}(a_{i_k},b_{i_k})F(s_1,\ldots,s_n)
$$
is a kind of the volume condition in $\mathbb R^k$.

\section{Basic Result - Lifting of $n$-Dimensional Spectral Resolutions}

The present section is one of the main parts of the paper. It gives a solution for lifting of $n$-dimensional spectral resolutions. These results will be applied in the next section to show how an $n$-dimensional spectral resolution entails the existence of the corresponding $n$-dimensional observable.

We begin to introduce some notations which allow us to handle with the process of lifting of $d$-cuboids and to control the volume conditions in a comfortable way.

Let $D\subseteq \{1,\ldots,n\}$, $d:=|D|$, and for each $i=1,\ldots,n$, let $a_i,b_i\in\mathbb{R}$ be such that $a_i<b_i$ whenever $i\in D$ and $a_i=b_i$ otherwise. Define $\mathcal{C}=\{(*_1,\ldots,*_n):*_i\in \{a_i,b_i\}\}$. We call $\mathcal{C}$ a $d$-\textit{cuboid}.
We will refer to the integer $d$ by $\dim \mathcal C := d$, to the set $D$ by $\mathrm{Dim}(\mathcal{C}):=D$ and to the $a_i$ ($b_i$, respectively) by $a_i^{\mathcal{C}}$ ($b_i^{\mathcal{C}}$, respectively). It is easy to see that any $d$-cuboid $\mathcal{C}$ has $2^d$ elements. Next, any $\mathcal{F}\subseteq\mathcal{C}$ which itself is an $e$-cuboid, $e\leq d$, is called an $e$-{\it face} of $\mathcal{C}$. If $e=0$, $\mathcal{F}$ is called a {\it vertex} of $\mathcal{C}$. Clearly vertices of $\mathcal{C}$ correspond to elements of $\mathcal{C}$ and by a slight abuse of notation we can identify them. It is also clear that the vertices of $\mathcal{C}$ can be partially ordered as they are elements of $\mathbb{R}^n$. We call $(b_1,\ldots,b_n)$ the top one or the first one, moreover, if a vertex $\alpha=(*_1,\ldots,*_n)\in\mathcal{C}$ has $a_i$'s for $m$ indices $i\in\mathrm{Dim}(\mathcal{C})$, we say $\alpha$ has an {\it order} $\mathrm{ord}_{\mathcal{C}}(\alpha):=m+1$ in $\mathcal{C}$ (i.e., the top vertex has an order $1$ and $(a_1,\ldots,a_n)$ has an order $\dim\mathcal{C}+1$). We say that some cuboid $\mathcal{D}$ is {\it inside} a cuboid $\mathcal{C}$, if for each $i\leq n$, $a^{\mathcal{C}}_i\leq a^{\mathcal{D}}_i\leq b^{\mathcal{D}}_i\leq b^{\mathcal{C}}_i$. In particular, every face of a cuboid $\mathcal{C}$ is inside $\mathcal{C}$.

As almost all essential steps in the process of the lifting will be achieved by an induction, the co-dimension of faces will be very important to us. For each $d$-cuboid $\mathcal{C}$ and $i\in\mathrm{Dim}(\mathcal{C})$, we define a $(d-1)$-cuboid $\partial_i\mathcal{C}:=\{(*_1,\ldots,*_n)\in\mathcal{C}: *_i=b_i\}$ and $\partial'_i\mathcal{C}:=\{(*_1,\ldots,*_n)\in\mathcal{C}: *_i=a_i\}$. Clearly $\partial_i\mathcal{C}$ and $\partial'_i\mathcal{C}$, $i=1,\ldots,n$, are all the $(d-1)$-faces of $\mathcal{C}$ and they are called {\it facets} of $\mathcal{C}$. Moreover, we say a facet is an {\it upper} ({\it lower}, respectively) {\it facet} of $\mathcal{C}$ if it arises by $\partial_i$ ($\partial'_i$, respectively) for some $i\in\mathrm{Dim}(\mathcal{C})$. We note that each facet is either an upper or a lower one. An easy but important observation is that whenever $i,j\in \mathrm{Dim}(\mathcal{C})$, $i\not= j$, we have
\begin{align}
\partial_i\circ\partial_j (\mathcal{C})=\partial_i(\mathcal{C})&\cap\partial_j(\mathcal{C})= \partial_j\circ\partial_i(\mathcal{C}), \label{eq:011}\\
\partial_i\circ\partial'_j (\mathcal{C})=\partial_i(\mathcal{C})&\cap\partial'_j(\mathcal{C})= \partial'_j\circ\partial_i(\mathcal{C}), \label{eq:012}
\\
\partial'_i\circ\partial'_j (\mathcal{C})=\partial'_i(\mathcal{C})&\cap\partial'_j(\mathcal{C})= \partial'_j\circ\partial'_i(\mathcal{C}). \label{eq:013}
\end{align}

Take a free Abelian group $\mathcal{A}_0$ generated by all cuboids in $\mathbb{R}^n$ and factorize it by the subgroup generated by elements
$$
\mathcal{C}-\partial_i(\mathcal{C})+\partial_i'(\mathcal{C}),\  i\in\Dim(\mathcal{C}).
$$
The resulting quotient Abelian group is denoted by $\mathcal{A}$. By an abuse of notation we will still refer to elements of $\mathcal{A}$ by cuboids. So in $\mathcal{A}$
$$
\mathcal{C}=\partial_i(\mathcal{C})-\partial'_i(\mathcal{C})
$$
holds for each cuboid $\mathcal{C}$ and $i\in\mathrm{Dim}(\mathcal{
C})$.

\begin{definition}\label{def:01}
Suppose we have cuboids $\mathcal{C}$, $\mathcal{C}_1$ and $\mathcal{C}_2$ of the same dimension. We say that a couple $\mathcal{C}_1$ and $\mathcal{C}_2$ is a \textit{splitting} of $\mathcal{C}$, if there is $i\in\Dim(\mathcal{C})$, such that  $\partial'_i(\mathcal{C}_1)=\partial'_i(\mathcal{C})$, $\partial_i(\mathcal{C}_1)=\mathcal{C}_1\cap\mathcal{C}_2= \partial'_i(\mathcal{C}_2)$ and $\partial_i(\mathcal{C}_2)=\partial_i(\mathcal{C})$. In other words, there is a real $c$, $a^{\mathcal{C}}_i< c< b^{\mathcal{C}}_i$, such that $\mathcal{C}_1$ shares with $\mathcal{C}$ all its coordinates unless $b^{\mathcal{C}_1}_i=c$ and  $\mathcal{C}_2$ shares with $\mathcal{C}$ all its coordinates unless $a^{\mathcal{C}_2}_i=c$.
\end{definition}

Observe that for the three cuboids from Definition~\ref{def:01}, we have in the group $\mathcal{A}$
\begin{equation}
\mathcal{C}=\partial_i(\mathcal{C}_2)-\partial'_i(\mathcal{C}_1) =(\partial_i(\mathcal{C}_2)-\partial'_i(\mathcal{C}_2))+ (\partial_i(\mathcal{C}_1)-\partial'_i(\mathcal{C}_1))= \mathcal{C}_1+\mathcal{C}_2.
\end{equation}

\begin{lemma}
Each cuboid $\mathcal{C}\in\mathcal{A}$ could be uniquely (up to order of summands) written in the form
\begin{equation}\label{eq:02}
\mathcal{C}=\sum_{\alpha\text{\ is\ a\ vertex\ in\ }\mathcal{C}} (-1)^{\mathrm{ord}_{\mathcal{C}}(\alpha)+1}\alpha.
\end{equation}
Hence, a vertex $\alpha$ occurs with $+1$ sign if it is of odd order in $\mathcal{C}$ and with $-1$ sign if it is of even order in $\mathcal{C}$. Consequently, $\mathcal{A}$ is in fact a free Abelian group generated by all the vertices  (elements of $\mathbb{R}^n$).
\end{lemma}

Let $L$ be a partial mapping from $\mathbb{R}^n$ to $\Gamma^e_a(G,u)$. Using~\eqref{eq:02}, we can extend $L$ to a group homomorphism $|\cdot|_L:\mathcal{A}_L\rightarrow G$, where $\mathcal{A}_L$ is the free subgroup of $\mathcal{A}$ generated by all vertices in $\Def(L)$. Hence $|\cdot|_L$ associates to each cuboid $\mathcal{C}$ having vertices in $\Def(L)$ its ``volume" element $|\mathcal{C}|_L \in G$.

In this section, we will suppose that we have fixed two Dedekind $\sigma$-complete unital po-groups $(G,u)$ and $(H,v)$ with interpolation and let $\pi:(G,u)\rightarrow (H,v)$ be a fixed homomorphism with the {\it lifting property} (LP), i.e. if, for each $L,U \subseteq \Gamma^e_a(G,u)$ finite and $h\in H$ such that $L\leq U$ and $\pi(L)\leq h\leq \pi(U)$, there is  $g\in G$ satisfying $\pi(g)=h$ and $L\leq g\leq U$. Note that (LP) implies surjectivity. In the same way we define (LP) for a homomorphism of effect algebras.

Let $F:\mathbb{R}^n\rightarrow \Gamma^e_a(H,v)$ be an $n$-dimensional spectral resolution and let $\pi:(G,u)\rightarrow (H,v)$ be a homomorphism of unital po-groups which satisfies the lifting property. We say, that a partial mapping $L:\mathbb{R}^n\rightarrow \Gamma(G,u)$ is a \textit{partial lift} of $F$, if $\pi\circ L=F$ and for each cuboid $\mathcal{C}\subseteq\Def(L)$ we have $|\mathcal{C}|_L\in\Gamma^e_a(G,u)$. That occurs if $0\leq L(\alpha)\leq u$ and $|\mathcal{C}|_L\geq 0$ for all vertices $\alpha$'s and all cuboids $\mathcal{C}$'s in the definition domain of $L$. We call the inequalities of the second type \textit{volume conditions}. We note that the volume condition of a cuboid does not imply volume conditions of its proper faces.

Now, let a partial lift $L$, a point $\alpha\in\mathbb R^n$, and $g \in \Gamma^e_a(G,u)$ be given. By $L'=L\cup \{(\alpha,g)\}$ we will understand that we would like to extend the definition domain of $L$ to the one of $L'=L\cup\{(\alpha,g)\}$ in such a way that $L'(\alpha)=g$.

\begin{lemma}\label{lem:01}
Let $\mathcal{C}$ be a $d$-cuboid, $\alpha$ be its vertex, $L$ be a partial lift defined on $\mathcal{C}\setminus \{\alpha\}$. For an extension $L'=L\cup \{(\alpha,g)\}$, $g\in\Gamma^e_a(G,u)$, we have: If $\alpha$ is of odd order in $\mathcal{C}$, then $\alpha-\mathcal{C}\in\mathcal{A}_L$ and the volume condition $|\mathcal{C}|_{L'}\geq 0$ holds iff $|\alpha-\mathcal{C}|_L\leq g$. If $\alpha$ is of even order in $\mathcal{C}$, then $\alpha+\mathcal{C}\in \mathcal{A}_L$ and the volume condition $|\mathcal{C}|_{L'}\geq 0$ holds iff $g\leq |\alpha+\mathcal{C}|_L$.
\end{lemma}

\begin{proof}
If $\alpha$ is of odd (even, respectively) order, it occurs in~\eqref{eq:02} with $+1$ ($-1$, respectively) sign, hence $\alpha$ is canceled in $\alpha-\mathcal{C}$ ($\alpha+\mathcal{C}$, respectively). Consider the odd case: $|\mathcal{C}|_{L'}\geq 0\Leftrightarrow |\mathcal{C}-\alpha|_{L'}+|\alpha|_{L'}\geq 0 \Leftrightarrow |\alpha|_{L'} \geq -|\mathcal{C}-\alpha|_{L'}\Leftrightarrow  g\geq |\alpha-\mathcal{C}|_{L'}=|\alpha-\mathcal{C}|_{L}$. The even case:  $|\mathcal{C}|_{L'}\geq 0\Leftrightarrow |\mathcal{C}+\alpha|_{L'}\geq |\alpha|_{L'} \Leftrightarrow |\mathcal{C}+\alpha|_{L}\geq g$.
\end{proof}

\begin{lemma}\label{lem:02}
Let $\mathcal{C}$ be a $d$-cuboid, $\alpha$ its top vertex, and $L_1$ a partial lift defined on $\mathcal{C}\setminus \{\alpha\}$. Then
\begin{equation}\label{ineq03}
|\alpha-\mathcal{C}|_{L_1}\geq 0.
\end{equation}
If $d\geq 1$, let $\beta$ be some vertex which is
in $\mathcal{C}$ of the second order and $L_2$ be a partial lift defined on  $\mathcal{C}\setminus\{\beta\}$. Denote $i\in\Dim(\mathcal{C})$ such that $\beta$ is the top vertex in $\partial_i'(\mathcal{C})$. Then
\begin{equation}\label{ineq04}
u\geq|\mathcal{C}+\beta|_{L_2}\geq |\beta- \partial_i'(\mathcal{C})|_{L_2} \geq 0.
\end{equation}
\end{lemma}

\begin{proof} We first prove~\eqref{ineq03} by an induction on $d$. The case $d=0$ is trivial. Suppose $d\geq 1$. We have $|\alpha-\mathcal{C}|_{L_1}=|\alpha-\partial_i(\mathcal{C})+ \partial'_i(\mathcal{C})|_{L_1}=|\alpha-\partial_i(\mathcal{C})|_{L_1}+ |\partial'_i(\mathcal{C})|_{L_1}$. As $\partial'_i(\mathcal{C})$ misses $\alpha$, $|\partial'_i(\mathcal{C})|_{L_1}\geq 0$ and by the induction hypothesis $|\alpha-\partial_i(\mathcal{C})|_{L_1}\geq 0$ as well.

Let us prove inequalities~\eqref{ineq04}:
We prove the first one by an induction on $d$. The case $d=1$ is trivial. Let $j\not= i$, i.e., $\beta\notin \partial'_j(\mathcal{C})$ and $\beta$ is in $\partial_j(\mathcal{C})$ of order $2$. Then $|\mathcal{C}+\beta|_{L_2}=|\partial_j(\mathcal{C})-\partial'_j(\mathcal{C})+ \beta|_{L_2}=|\partial_j(\mathcal{C})+\beta|_{L_2}- |\partial'_j(\mathcal{C})|_{L_2}\leq |\partial_j(\mathcal{C})+\beta|_{L_2}$, and the last one is $\leq u$ by the induction hypothesis. The next inequality follows by: $|\beta+\mathcal{C}|_{L_2}\geq |\beta- \partial_i'(\mathcal{C})|_{L_2} \Leftrightarrow |\beta+\mathcal{C}|_{L_2}-|\beta- \partial_i'(\mathcal{C})|_{L_2}\geq 0\Leftrightarrow |\mathcal{C}|_{L_2}+|\partial_i'(\mathcal{C})|_{L_2}\geq 0$. But  $|\mathcal{C}+\partial_i'(\mathcal{C})|_{L_2}=|\partial_i(\mathcal{C})|_{L_2}\geq 0$ as $L_2$ is defined on the cuboid $\partial_i(\mathcal{C})$.

Finally, $|\beta-\partial'_i(\mathcal{C})|_L\geq 0$ follows from the already proved inequality~\eqref{ineq03} ($\beta$ is the top vertex in $\partial'_i(\mathcal{C})$).
\end{proof}

\begin{lemma}\label{lem:ortho}
Let $\mathcal{C}$ be a cuboid and $\mathcal F(\mathcal C)$ be a collection of its facets such that, for each $i\in \Dim(\mathcal{C})$, the facet $\partial_i(\mathcal{C})$ or  $\partial'_i(\mathcal{C})$ does not belong to $\mathcal F(\mathcal C)$. Then, for each sub-cuboid $\mathcal{D}\subseteq\mathcal{C}$ which satisfies $\mathcal{D}\subseteq\bigcup \mathcal F(\mathcal C)$, there is some $\mathcal{F}\in \mathcal F(\mathcal C)$ such that $\mathcal{D}\subseteq\mathcal{F}$. Consequently, if $L$ is a mapping $L:\bigcup\mathcal{F}\rightarrow G$ whose restriction to any $\mathcal{F}\in \mathcal F(\mathcal C)$ is a partial lift, then $L$ itself is a partial lift.
\end{lemma}

\begin{proof}
Suppose a cuboid $\mathcal{D}\subseteq \bigcup \mathcal F(\mathcal C)$ which is not a sub-cuboid of any $\mathcal{F}\in \mathcal F(\mathcal C)$. Take any $\mathcal{F}\in \mathcal F(\mathcal C)$ and denote $i_{\mathcal{F}}$ the unique integer such that $\Dim(\mathcal{F})\cup\{i_\mathcal{F}\}=\Dim(\mathcal{C})$. We
have either $i_\mathcal{F}\in\Dim(\mathcal{D})$ or $i_\mathcal{F}\notin\Dim(\mathcal{D})$ and $a^\mathcal{D}_{i_\mathcal{F}}=b^\mathcal{D}_{i_\mathcal{F}}\not= a^\mathcal{F}_{i_\mathcal{F}}=b^\mathcal{F}_{i_\mathcal{F}}$. Since otherwise $\mathcal{D}\subseteq\mathcal{F}$. Consequently, in $\mathcal{D}$ there is a vertex $\alpha$ such that $a^\alpha_{i_{\mathcal{F}}}\not= a^\mathcal{F}_{i_
\mathcal{F}}$ for each $\mathcal{F}\in \mathcal F(\mathcal C)$. Note that $i_{\mathcal{F}_1}\not= i_{\mathcal{F}_2}$, whenever $\mathcal{F}_1\not=\mathcal{F}_2$. So $\alpha$ is not a vertex of any of $\mathcal{F}\in \mathcal F(\mathcal C)$, which is a contradiction.
\end{proof}

\begin{lemma}\label{lem:ortho2}
Let $\mathcal{C}_1,\mathcal{C}_2$ be a splitting of a cuboid $\mathcal{C}$, $i,c$ be as in Definition~\ref{def:01} and a cuboid $\mathcal{D}\subseteq \mathcal{C}_1\cup\mathcal{C}_2$, but $\mathcal{D}\nsubseteq\mathcal{C}$. Then $\mathcal{D}\subseteq\mathcal{C}_1$ or $\mathcal{D}\subseteq\mathcal{C}_2$. Consequently, if $L:\mathcal{C}_1\cup\mathcal{C}_2\rightarrow \Gamma^e_a(G,u)$ and the restrictions of $L$ to $\mathcal{C}_1$ and to $\mathcal{C}_2$ are both partial lifts, then $L$ is a partial lift as well.
\end{lemma}

\begin{proof}
The $i$-th coordinates of the vertices of $\mathcal{D}$ belong to $\{a^{\mathcal{C}_i},c\}$ or $\{c,b^{\mathcal{C}_i}\}$, the first case implies $\mathcal{D}\subseteq \mathcal{C}_1$, the other one implies $\mathcal{D}\subseteq \mathcal{C}_2$.
\end{proof}

\begin{lemma}\label{lem:liftCuboid}
Let $\mathcal{C}$ be a cuboid in $\mathbb{R}^n$ and $L$ be a partial lift in $\mathcal{C}$ such that the definition domain of $L$ equals one of the following
\begin{enumerate}[{\rm (i)}]
\item\label{c1} $\emptyset$,
\item\label{c2} one lower facet, that is $\partial'_i(\mathcal{C})$ for some $i\in\mathrm{Dim}(\mathcal{C})$,
\item\label{c3} a union of one upper facet $\partial_i(\mathcal{C})$, $i\in\mathrm{Dim}(\mathcal{C})$, and a collection of lower facets $\partial'_j(\mathcal{C})$, $j\in J$, for some $J\subseteq \mathrm{Dim}(\mathcal{C})\setminus \{i\}$.
\end{enumerate}
Then we can extend $L$ on the whole cuboid $\mathcal{C}$.
\end{lemma}

Note that if $\mathcal{C}$ is a rectangle, the case when all vertices up to the top one are lifted is excluded.

\begin{proof}
We will use an induction on $\dim\mathcal{C}$. The case $\dim \mathcal{C}$ equals $0$ or $1$ is trivial. Suppose $\dim \mathcal{C}\geq 2$ and the case~\eqref{c1}. Take any upper facet $\mathcal{F}$ of $\mathcal{C}$ and use the case~\eqref{c1} of the induction hypothesis to define $L$ on $\mathcal{F}$. We have arrived at the case~\eqref{c3}.

If the case~\eqref{c2} holds, pick any $j\in\mathrm{Dim}(\mathcal{C})\setminus
\{i\}$. If we extend $L$ on $\mathcal{F}=\partial_j(\mathcal{C})$, we will arrive at the case~\eqref{c3} again. Since $\partial'_i(\mathcal{C})\cap\mathcal{F}=\partial'_i(\mathcal{F})$ is a lower facet of $\mathcal{F}$, we can use the case~\eqref{c2} of the induction hypothesis to $\mathcal{F}$. So we obtain a partial lift $L'$ on $\mathcal{F}$ and by Lemma~\ref{lem:ortho}, $L\cup L'$ is a partial lift.

Hence, it remains to prove the case~\eqref{c3}. Take any lower facet $\mathcal{F}=\partial_k'(\mathcal{C})$. If $j\notin J$, we extend $L$ on $\mathcal{F}$: Since the upper facet $\partial_i(\mathcal{C})$ of $\mathcal{C}$ intersects $\mathcal{F}$ in an upper facet of $\mathcal{F}$ and similarly each lower facet $\partial'_j(\mathcal{C})$, $j\in J$, intersects $\mathcal{F}$ in a lower facet of $\mathcal{F}$, we can by the case~\eqref{c3} of the induction hypothesis and Lemma~\ref{lem:ortho} extend $L$ on $\mathcal{F}$.

Hence, we can assume $J=\mathrm{Dim}(\mathcal{C})\setminus\{i\}$. That is, the only vertex that remains to lift is $\beta:=(b_1,\ldots,b_{i-1},a_i,b_{i+1},\ldots,b_d)$. Note that $\beta$ is of order $2$ in $\mathcal{C}$. We have to find an extension $L'=L\cup\{(\beta,g)\}$, $g\in \Gamma^e_a(G,u)$, such that
\begin{equation}\label{ineq01}
|\mathcal{C}|_{L'}\geq 0,
\end{equation}
\begin{equation}\label{ineq02}
|\partial'_i(\mathcal{C})|_{L'}\geq 0.
\end{equation}
We claim, that if these two volume conditions hold, then all the volume conditions hold in $\mathcal{C}$. At the first, we note that $|\partial_j(\mathcal{C})|_{L'}\geq 0$ holds for each $j\in\mathrm{Dim}(\mathcal{C})\setminus\{i\}$ (the case $j=i$ holds by assumptions): $|\partial_j(\mathcal{C})|_{L'}=|\mathcal{C}|_{L'}+ |\partial_j'(\mathcal{C})|_{L'}\geq 0$, since $|\partial_j'(\mathcal{C})|_{L'}=|\partial_j'(\mathcal{C})|_{L}\geq 0$ and~\eqref{ineq01}. Hence, the volume condition holds in each facet of $\mathcal{C}$. Next let $\mathcal{F}$ be a face of $\mathcal{C}$ which contains $\beta$ and is in $\mathcal{C}$ of co-dimension $e\geq 2$. It follows $\mathcal{F}=\Theta_{i_1}\circ\cdots\circ \Theta_{i_e}(\mathcal{C})$, where each $\Theta_{i_k}\in\{\partial_{i_k},\partial'_{i_k}\}$. Since $\mathcal{F}$ contains $\beta$ which is a vertex of order $2$ in $\mathcal{C}$, the number of $k$, so that $\Theta_{i_k}=\partial'_{i_k}$ is at most $1$. In particular, there is $k\leq e$ such that $\Theta_{i_k}=\partial_{i_k}$ and so $\mathcal{F}$ is of the form $\partial_{i_k}(\mathcal{F}_0)$ (by commutativity~\eqref{eq:012}). Now we use the formula
$$
|\mathcal{F}|_{L'}=|\partial_{i_k}(\mathcal{F}_0)|_{L'}= |\mathcal{F}_0|_{L'}+|\partial'_{i_k}(\mathcal{F}_0)|_{L'}.
$$
Since $|\partial'_{i_k}(\mathcal{F}_0)|_{L'}=|\partial'_{i_k}(\mathcal{F}_0)|_L\geq 0$ (as it misses $\beta$) it is enough to prove the volume condition for the face with less co-dimension in $\mathcal{C}$. In other words, we can use an induction on co-dimension to finish the proof of the claim.

As $\beta$ occurs in $\mathcal{C}$ as an element of order $2$ and in $\partial'_i(\mathcal{C})$ as an element of order $1$, by Lemma~\ref{lem:01} we can replace inequalities~\eqref{ineq01},~\eqref{ineq02} by equivalent conditions
$$|\mathcal C+\beta|_L\geq g\ \ \&\ \ g\geq |\beta-\partial'_i(\mathcal{C})|_L.$$
We finish by applying (LP) with bounds $\{u,|\mathcal C+\beta|_L\}\geq \{|\beta-\partial'_i(\mathcal{C})|_L,0\}$ which are consistent by Lemma~\ref{lem:02}.
\end{proof}

In the following lemma $\mathcal{S}$ figures as a section of $\mathcal{C}$ orthogonal to the axis $i$.

\begin{lemma}\label{lem:section}
Let $\mathcal{C}_1$ and $\mathcal{C}_2$ be a splitting of a $d$-cuboid $\mathcal{C}$, $i,c$ be as in Definition~\ref{def:01} and denote $\mathcal{S}$ the $(d-1)$-cuboid $\partial_i(\mathcal{C}_1)=\mathcal{C}_1\cap\mathcal{C}_2= \partial'_{i}(\mathcal{C}_2)$. Next let $L$ be a partial lift defined on all the vertices of $\mathcal{C}$ and $\bigcup \mathcal F(\mathcal S)$, where $\mathcal F(\mathcal S)$ is the collection (possibly empty) of facets of $\mathcal S$ such that for each $j\in\Dim(\mathcal{S})$, $\partial_j(\mathcal{S})$ or $\partial'_j(\mathcal{S})$ does not belong to $\mathcal F(\mathcal S)$ and in $\mathcal F(\mathcal S)$ is at most one upper facet of $\mathcal{S}$.
Then there is a partial lift which extends $L$ on the whole $\mathcal{S}$.
\end{lemma}

\begin{proof}
We first maximalize the collection $\mathcal F(\mathcal S)$. Let $j\in\Dim(\mathcal{S})$ be such that neither $\partial_j(\mathcal{S})$ nor  $\partial'_j(\mathcal{S})$ belongs to $F$. We like to extend $L$ on $\partial'_j(\mathcal{S})$. As each upper facet in $\mathcal F(\mathcal S)$ intersects $\partial'_j(\mathcal{S})$ in an upper facet of $\partial_j(\mathcal{S})$ and each lower facet in $\mathcal F(\mathcal S)$ intersects $\partial_j(\mathcal{S})$ in a lower facet of $\partial'_j(\mathcal{S})$, we can use the induction hypothesis, where we take $\partial'_j(\mathcal{C})$ for $\mathcal{C}$ and $\partial'_j (\mathcal F(\mathcal S))$ for $\mathcal F(\mathcal S)$. So we obtain a partial lift $L'$ defined on $\partial'_j(C)\cup\partial'_j(\mathcal{S})$.

We have to prove $L\cup L'$ is a partial lift. For this it is enough to realize that, whenever for some cuboid $\mathcal{D}$, we have $\mathcal{D}\subseteq (\partial'_j(\mathcal C)\cup\partial'_j(\mathcal{S}))\cup (\mathcal{C}\cup (\bigcup \mathcal F(\mathcal S)))$ (definition domain of $L'\cup L$), then $\mathcal{D}\subseteq \partial'_j(\mathcal C)\cup\partial'_j(\mathcal{S})$ (definition domain of $L'$) or  $\mathcal{D}\subseteq\mathcal{C}\cup (\bigcup \mathcal F(\mathcal S))$ (definition domain of $L$). If $\mathcal{D}\subseteq \mathcal{C}$, we are done. Otherwise by Lemma~\ref{lem:ortho2} we have $\mathcal{D}\subseteq \mathcal{C}_1$
or $\mathcal{D}\subseteq \mathcal{C}_2$. We can treat the both cases in similar fashion. We present proof of the one where $\mathcal{D}\subseteq \mathcal{C}_1$. If $i\in\Dim(\mathcal{D})$, then $\partial_i(\mathcal{D})\subseteq (\bigcup \mathcal F(\mathcal S))\cup \partial'_j(\mathcal{S})$. By Lemma~\ref{lem:ortho} either holds: $\partial_i(\mathcal{D})\subseteq \bigcup  \mathcal F(\mathcal S)$, and then $\mathcal{D}\subseteq \mathcal{C}\cup (\bigcup \mathcal F(\mathcal S))$, or $\partial_i(\mathcal{D})\subseteq \partial'_j(\mathcal S)$, and then $\mathcal{D}\subseteq \partial'_j(\mathcal{C})\cup\partial'_j(\mathcal{S})$. If $i\notin \Dim(\mathcal{C})$, then in the above deduction replace $\partial_i(\mathcal{C})$ with $\mathcal{C}$.

So we can assume for each $j\in\Dim(\mathcal{S})$ either $\partial_j(\mathcal{S})\in \mathcal F(\mathcal S)$ or  $\partial'_j(\mathcal{S})\in \mathcal F(\mathcal S)$. Two cases possibly occur: (i) all facets in $\mathcal F(\mathcal S)$ are lower and the top vertex $\beta$ of $\mathcal{S}$ is the only one which remains to lift, or (ii) there is $k$, such that $\partial_k(\mathcal{S})\in \mathcal F(\mathcal S)$ and $\gamma$ the only vertex  which remains to lift has order $2$ in $\mathcal{S}$.

Assume the case (ii). For some $g\in \Gamma^e_a(G,u)$ define $L'=L\cup\{(\gamma,g)\}$. We note that $\gamma$ has order $2$ in $\mathcal{C}_1$ and in $\mathcal{S}$ and order $3$ in $\mathcal{C}_2$. To assure $L$ is a partial lift on $\mathcal{C}_1$, the following volume condition has to hold (by analogy with~\eqref{ineq01}) and~\eqref{ineq02}:
\begin{align}
|\mathcal{C}_1|_{L'}\geq 0, \label{vc1}\\
|\partial'_k(\mathcal{C}_1)|_{L'}\geq 0.  \label{vc2}
\end{align}
The case of $\mathcal{C}_2$ requires these volume conditions:
\begin{align}
|\mathcal{C}_2|_{L'}\geq 0,  \label{vc3}\\
|\partial'_k(\mathcal{C}_2)|_{L'}\geq 0,  \label{vc4}\\
|\partial'_i(\mathcal{C}_2)|_{L'}\geq 0, \label{vc5}\\
|\partial'_k\circ \partial'_i(\mathcal{C}_2)|_{L'}\geq 0.  \label{vc6}
\end{align}
We claim inequations~(\ref{vc3})--(\ref{vc6}) are sufficient. Let $\mathcal{F}$ be any face of $\mathcal{C}_2$; it has a form
\begin{equation}\label{eq:03}
\mathcal{F}=\Theta_{i_1}\circ\cdots\circ \Theta_{i_e}(\mathcal{C}_2),\quad e\leq\dim\mathcal{S},
\end{equation} where each $\Theta_{i_k}\in\{\partial_{i_k},\partial'_{i_k}\}$. We first prove by an induction on the number of occurrence of $\partial_{i_k}$'s in~\eqref{eq:03}, that it is enough to treat the cases when each $\Theta_{i_k}=\partial'_{i_k}$. Since whenever $\mathcal{F}$ has form $\mathcal{F}=\partial_{j}(\mathcal{F}_0)$, then $|\mathcal{F}|_{L'}=|\partial'_{j}(\mathcal{F}_0)|_{L'}+|\mathcal{F}_0|_{L'}$. If the two summand are $\geq 0$, so is the left hand side. However, if $\mathcal{F}=\partial'_{i_1}\circ\cdots\circ \partial'_{i_e}(\mathcal{C}_2)$ contains $\gamma$, each $i_k$'s belongs to $\{i,k\}$, as $\gamma$ is the top vertex of $\partial'_{k}\circ\partial'_{i}(\mathcal{C}_2)$. So~(\ref{vc3})--(\ref{vc6}) are all the volume conditions that matter.

Note that~\eqref{vc5} and~\eqref{vc6} are volume conditions for sub-cuboids of $\mathcal{C}_1$ (since $\partial'_i(\mathcal{C}_2)=\partial_i(\mathcal{C}_1)$), and hence they follow from~\eqref{vc1} and~\eqref{vc2}. To assure $L'$ is a partial lift, inequalities~(\ref{vc1})--(\ref{vc4}) give us, according to Lemma~\ref{lem:01}, the following bounds:
\begin{align}
g\leq |\gamma+\mathcal{C}_1|_L,\label{bound1}\\
|\gamma-\partial'_k(\mathcal{C}_1)|_L\leq g,\label{bound2}\\
|\gamma-\mathcal{C}_2|_L\leq g,\label{bound3}\\
g\leq |\gamma+\partial'_k(\mathcal{C}_2)|_L.\label{bound4}
\end{align}
We already know by Lemma~\ref{lem:02} that~\eqref{bound1} and~\eqref{bound2} are consistent. So are~\eqref{bound1} and~\eqref{bound3}, since: $|\gamma-\mathcal{C}_2|_L\leq |\gamma+\mathcal{C}_1|_L\Leftrightarrow |\mathcal{C}_1+\mathcal{C}_2|_L\geq 0\Leftrightarrow |\mathcal{C}|_L\geq 0$. Next $|\gamma-\partial'_k(\mathcal{C}_1)|_L\leq |\gamma+\partial'_k(\mathcal{C}_2)|_L \Leftrightarrow 0\leq |\partial'_k(\mathcal{C}_1)+\partial'_k(\mathcal{C}_2)|_L= |\partial'_k(\mathcal{C})|_L$. Finally, $|\gamma-\mathcal{C}_2|_L\leq |\gamma+\partial'_k(\mathcal{C}_2)|_L \Leftrightarrow 0\leq |C_2+\partial'_k(\mathcal{C}_2)|_L=|\partial_k(\mathcal{C}_2)|_L$. Hence, we can apply the lifting property to obtain the desired lift in $\gamma$. The lift necessary belongs to $\Gamma^e_a(G,u)$ due to inequalities~\eqref{ineq04} in Lemma~\ref{lem:02}.

Next assume the (easier) case (i). Again define $L'=L\cup\{(\beta,g)\}$ for some $g\in G$. By the analogous arguments as in the case (ii), it is enough to assure volume conditions
\begin{align}
|\mathcal{C}_1|_{L'}\geq 0, \label{vc7}\\
|\mathcal{C}_2|_{L'}\geq 0, \label{vc8}\\
|\partial'_i(\mathcal{C}_2)|_{L'}\geq 0.\label{vc9}
\end{align}
But $\partial'_i(\mathcal{C}_2)=\partial_i(\mathcal{C}_1)$, hence the inequality~\eqref{vc7} implies~\eqref{vc9}. By Lemma~\ref{lem:01}, $$|\beta-\mathcal{C}_1|_L\leq g\leq |\beta+\mathcal{C}_2|_L$$ is equivalent condition to~\eqref{vc7} and~\eqref{vc8}. These bounds are consistent as $|\beta-\mathcal{C}_1|_L\leq |\beta+\mathcal{C}_2|_L\Leftrightarrow 0\leq |\mathcal{C}_1+\mathcal{C}_2|=|\mathcal{C}|_L$. Moreover, Lemma~\ref{lem:02}, inequalities~\eqref{ineq03} and~\eqref{ineq04} guarantee $0\leq |\beta-\mathcal{C}_1|_L\leq |\beta+\mathcal{C}_2|_L\leq u$. Hence, we can finish the proof by application of the lifting property.
\end{proof}

\begin{lemma}\label{lem:refinement}
Let $\mathcal{F}(\mathcal{C})$ be a collection of facets of a cuboid $\mathcal{C}$ such that,
for each $i\in\Dim(\mathcal{C})$, one of $\partial_i(\mathcal{C}),\partial'_i(\mathcal{C})$
 does not belong to $\mathcal{F}(\mathcal{C})$ and it
contains at most one upper facet. Next let $J\subseteq \Dim(\mathcal{C})$ (in these directions we want to find a refinement) and, for each $j\in J$, there is a real $c_j$, $a^\mathcal{C}_j<c_j<b^\mathcal{C}_j$. Denote $$X=\{(r_1,\ldots,r_n)\colon r_i\in\{a^\mathcal{C}_j,c_j,b^\mathcal{C}_j\} \text{\ for\ }j\in J, r_i\in\{a^\mathcal{C}_j,b^\mathcal{C}_j\}\text{\ for\ }j\notin J\}$$
and $Y\subset X$ be those vertices in $X$ which are inside some facet in $\mathcal{F}(\mathcal{C})$.  Then each partial lift $L$ on $\mathcal{C}\cup Y$ could be extended to a partial lift on $X$.
\end{lemma}

\begin{proof}
We prove the lemma by an induction on $|J|$. The case $J=\{j\}$ is the same as previous Lemma~\ref{lem:section}. Suppose $J=\{j\}\cup J'$, where $|J'|<|J|$, and for simplicity suppose $j=1$. Similarly as in Lemma~\ref{lem:section} we have the splitting of $\mathcal{C}$ (along the j-th (first) coordinate) to $\mathcal{C}_1$ and $\mathcal{C}_2$ and we set $\mathcal{S}=\mathcal{C}_1\cap \mathcal{C}_2$.

We first apply the induction hypothesis where we take for $J$ the set $\{1\}$ and for $c_1$ the $c_1$. We obtain a lift $L_1$ defined on $\mathcal{C}\cup \mathcal{S}$. Before we processed further we have to verify $L_1\cup L$ is a partial lift.  As usual, we prove that whenever some cuboid $\mathcal{D}\subseteq (\mathcal{C}\cup \mathcal{S})\cup (\mathcal{C}\cup Y)$, then $\mathcal{D}\subseteq \mathcal{C}\cup \mathcal{S}$ or $\mathcal{D}\subseteq\mathcal{C}\cup Y$. On the way of contradiction assume $\alpha,\beta$ are two vertices in $\mathcal{D}$ such that $\alpha\in S\setminus (\mathcal{C}\cup Y)$ and $\beta\in Y\setminus (\mathcal{C}\cup \mathcal{S})$. In coordinates $\alpha=(c_1,*_2,\ldots,*_n)$ and $\beta=(\star_1,\ldots,c_i,\ldots,\star_n)$, where $*_k$'s belong to $\{a^{\mathcal{C}}_k,a^{\mathcal{C}}_k\}$, $\star_k$'s belong to $\{a^{\mathcal{C}}_k,c_k,a^{\mathcal{C}}_k\}$  and $i\not= 1$. As $\mathcal{D}$ is a cuboid, it contains also the vertex $\gamma=(c_1,*_2,\ldots,*_{i-1},c_i,*_{i+1},\ldots,*_n)$ (which shares with $\alpha$ all the coordinates until the $i$-th). Now $\gamma\notin \mathcal{C}\cup\mathcal{S}$ (because of it's $i$-th coordinate) so $\gamma\in Y$ and, by definition of $Y$, there is a facet $\partial_k(\mathcal{C})$ ($\partial'_k(\mathcal{C})$, resp.) in $\mathcal{F}(\mathcal{C})$ such that $\gamma$ is inside  the facet. That occurs iff $*_k=b^\mathcal{C}_k$ ($*_k=a^\mathcal{C}_k$, resp.). But in this case $\alpha$ is inside the facet as well and so it belongs to $Y$, which is a contradiction.

We like to apply the induction hypothesis to both cuboids $\mathcal{C}_1$ and $\mathcal{C}_2$ with $J'$ to obtain partial lifts $L_2,L_3$, but we have to be careful in what order. If $\partial_1(\mathcal{C})\in\mathcal{F}(\mathcal{C})$ (i.e., it is the only upper facet in $\mathcal{F}(\mathcal{C})$), then we first apply the induction hypothesis to $\mathcal{C}_2$, where to $\mathcal{F}(\mathcal{C}_2)$ we add $\partial_1(\mathcal{C}_2)$ ($=\partial_1(\mathcal{C})$) and moreover we add $\partial'_i(\mathcal{C}_2)$ whenever $\partial'_i(\mathcal{C})$ belongs to $\mathcal{F(\mathcal{C})}$. Then we apply the induction hypothesis to $\mathcal{C}_1$ with $J'$ and $\mathcal{F}(\mathcal{C}_1)$ defined in analogy with $\mathcal{F}(\mathcal{C}_2)$ (note that $\mathcal{C}_1\cap\mathcal{C}_2=\partial_j(\mathcal{C}_1)$ is included). In the case $\partial_1(\mathcal{C})\not\in\mathcal{F}(\mathcal{C})$, we first apply the induction hypothesis to $\mathcal{C}_1$, since otherwise, in the case $\partial'_j\mathcal{C}=\partial'_1(\mathcal{C}_1)\in\mathcal{F}(\mathcal{C})$, $\partial'_1(\mathcal{C}_1)$ and $\mathcal{C}_1\cap\mathcal{C}_2=\partial_1(\mathcal{C}_1)$, there  would be a pair of opposite facets of $\mathcal{C}_1$ which we would have to include to $\mathcal{F}(\mathcal{C}_1)$.

We have obtained a mapping $L_4=L_2\cup L_3$, which satisfies the volume condition for each cuboid having vertices in $X$ and which is inside $\mathcal{C}_1$ or $\mathcal{C}_2$. Each
cuboid $\mathcal{D}\subseteq X$ could be in the obvious way split into $\mathcal{D}_1$
and $\mathcal{D}_2$, where $\mathcal{D}_1$ is inside $\mathcal{C}_1$,
$\mathcal{D}_2$ is inside $\mathcal{C}_2$ and in $\mathcal{A}$ we have
$\mathcal{D}=\mathcal{D}_1+\mathcal{D}_2$. Consequently
$|\mathcal{D}|_{L_4}=|\mathcal{D}_1|_{L_4}+|\mathcal{D}_2|_{L_4}\geq 0$.
\end{proof}

\begin{proposition}\label{pr:main}
Let $(G,u)$ and $(H,v)$ be unital Dedekind monotone  $\sigma$-complete po-groups and let $\pi:(G,u)\rightarrow (H,v)$ be a homomorphism with (LP). Let $F$ be an $n$-dimensional spectral resolution on $(H,v)$. Then there is a countable and dense subset $D\subset \mathbb{R}^n$ and a partial lift $L$ of $F$ which is defined on $D$.
\end{proposition}

\begin{proof}
Define for each $l\in\mathbb{N}_0$ the set $\mathcal{U}_l$ of all $n$-cubes with coordinates in $\frac{1}{2^l}\mathbb{Z}$ which have edges of length $\frac{1}{2^l}$ (e.g., all $n$-cuboids $\mathcal{C}$, for which $a^{\mathcal{C}}_j,b^{\mathcal{C}}_j\in\frac{1}{2^l}\mathbb{Z}$ and $b^{\mathcal{C}}_j-a^{\mathcal{C}}_j=\frac{1}{2^l}$ for each $j$, $1\le j\leq n$). Our strategy is as follows: We will inductively construct a sequence of partial lifts $L_0\subset L_1\subset \cdots$ such that each $L_l$ is defined on $D_l:=\bigcup_{i\le l}\mathcal{U}_i$. The set $D:=\bigcup_{l\geq 0}D_l$ is countable and dense in $\mathbb{R}^n$ and a partial mapping $\bigcup_{l\geq 0}L_l$ will be the desired partial lift.

\vspace{3mm}
\textit{Claim } {\it There is a list $\mathcal{E}_1,\mathcal{E}_2,\ldots$ of all the $n$-cubes in $\mathcal{U}_l$, $l\geq 0$, so that the following property $(*)$ holds: For each $m\in\mathbb{N}$, we have $\mathcal{E}_m\cap (\bigcup_{j< m} \mathcal{E}_j)$ is of one of following types: {\rm (i)} empty, {\rm (ii)} one lower facet of $\mathcal{E}_m$, {\rm (iii)} union of one upper facet $\partial_k(\mathcal{E}_m)$ and a collection of lower facets $\partial'_j(\mathcal{E}_m)$, where $J\subseteq\{1,\ldots,n\}\setminus \{k\}$.}
\vspace{2mm}

Using Claim, it is rather easy to construct the $L_l$'s: Suppose $\mathcal{C}_1,\mathcal{C}_2,\ldots$ is the list from Claim of all the elements in $\mathcal{U}_0$. We first apply Lemma~\ref{lem:liftCuboid} to $\mathcal{C}_0$, then to $\mathcal{C}_1$ and so on inductively we define a lift $L_0$. In next step we like to extend $L_0$ on each point with coordinates in $\frac{1}{2}\mathbb{Z}$, that is on $D_1=\bigcup \mathcal{U}_1$. To achieve this, we use another induction process. At the first step, we use Lemma~\ref{lem:refinement} to the first $n$-cube $\mathcal{C}_1\in\mathcal{U}_0$: We leave the set $\mathcal F(\mathcal C_1)$ from the statement empty and we set $J=\{1,\ldots,n\}$ and $c_j=\frac{a^{\mathcal{C}_j}+b^{\mathcal{C}_j}}{2}$, $j\in J$, so the points in $X$ are exactly the ones of $D_1$ we want to lift. In the $m$-th step we use Lemma~\ref{lem:refinement} in a similar way, the only difference is that we cannot set $\mathcal F(\mathcal C_1)$ empty, since points inside some facets of $\mathcal{C}_m$ have already been lifted, but the collection of these facets is, by the property $(*)$, in a convenient form.

In the next step we use an analogous induction process, but with $\mathcal{U}_2$ in place of $\mathcal{U}_1$. So we find lifts in all points with coordinates in $\frac{1}{2^2}\mathbb{Z}$. Similarly we find partial lifts $L_3\subset L_4\subset\cdots$. The desired dense set and partial lift are $D:=\bigcup_{m\in\mathbb{N}} D_m$ and $L:=\bigcup_{m\in\mathbb{N}} L_m$. The construction guarantees for each $L_m$ the volume condition $|\mathcal{C}|_{L_m}\geq 0$, only for $\mathcal{C}$ being a face of some cuboid in $\mathcal{U}_m$. However, it is clear that each cuboid with vertices in $D_m$ has as an element of $\mathcal{A}$ a decomposition $\mathcal{C}=\mathcal{D}_1+\cdots +\mathcal{D}_k$ such that all $\mathcal{D}_i$'s are already faces of some cuboids in $\mathcal{U}_m$.
Finally, each volume condition $|\mathcal{C}|_L\geq 0$ holds, where $\mathcal{C}\subset D$, since $\mathcal{C}\subset D_m$ for some $m$, as $\mathcal{C}$ has only a finite number of vertices.

\begin{proof}[Proof of Claim]
It is enough to prove, that given any $n$-cuboid $\mathcal{C}$ (arbitrary large) and a list $\mathcal{E}_1,\ldots,\mathcal{E}_{m_{\mathcal{C}}}$ of all unit cubes inside $\mathcal{C}$, then whenever we enlarge $\mathcal{C}$ in some of the $2n$ directions by one to get $\mathcal{C}'$, we can add at the end of the list all the new unit cuboids inside $\mathcal{C}'$ but not inside $\mathcal{C}$, so that $(*)$
still holds. An obvious inductive construction then gives the desired list.

So suppose we are given a cuboid $\mathcal{C}$ and the list $\mathcal{E}_1,\ldots,\mathcal{E}_{m_{\mathcal{C}}}$ and we want to move some lower
facet $\partial'_i{\mathcal{C}}$.
Denote by $\mathcal{D}$ the extra part of $\mathcal{C}'$ (i.e., $\mathcal{C}$ and $\mathcal{D}$ is a splitting of $\mathcal{C}'$). Observe that each unit cuboid $\mathcal{D}'$ inside $\mathcal{D}$ intersects the union $\bigcup_{j=1}^{m_\mathcal{C}}\mathcal{E}_k$ in exactly one of its upper facet: $\partial_i(\mathcal{D}')$. Hence, we basically need to find a list $\mathcal{D}_1,\ldots,\mathcal{D}_{m_{\mathcal{D}}}$ of all the unit cuboids inside $\mathcal{D}$ such that the intersection of any $D_k$ with the union of the previous ones is a collection of lower facets of $\mathcal{D}_k$. Then the list $\mathcal{E}_1,\ldots,\mathcal{E}_{m_{\mathcal{C}}},\mathcal{D}_1,\ldots,\mathcal{D}_{m_{\mathcal{D}}}$ will satisfy condition $(*)$. But finding the desired list is easy: Any ordering of the axis gives us a lexicographic ordering of the unit cuboids, which restricts to desired list $\mathcal{D}_1,\ldots,\mathcal{D}_{m_{\mathcal{D}}}$.

Next suppose we want to move an upper facet $\partial_i(\mathcal{C})$. Let $\mathcal{D}$ have the same meaning as above and set $m:=b_l^{\mathcal{D}}-a_l^{\mathcal{D}}$ for some $l\in\{1,\ldots,n\}\setminus\{i\}$. In this situation, we need to find a list $\mathcal{D}_1,\ldots,\mathcal{D}_{m_{\mathcal{D}}}$ of all the unit cuboids inside $\mathcal{D}$ such that the intersection of any $\mathcal D_k$ with the union of the previous ones is (i) empty or (ii) a union of some upper facet $\partial_o(\mathcal{D}_k)$ and a collection of lower facets $\partial'_j(\mathcal{D}_k)$, $j\in J$, where $J\subseteq \{1,\ldots,n\}\setminus\{i,o\}$. We prove this using an induction on the dimension $n$. The case $n=2$ is very easy to deal with: The cuboids in concern form simply a column of height $m$, so we list them from the top one to to bottom one. For $n\geq 3$, we can with respect to the $l$-axis dived the unit cuboids inside $\mathcal{D}$ into $m$ floors. At the first step we like to find such a list for the top floor, but this is a one dimensional less situation, so we can apply the induction hypothesis. Next we like to add the unit cuboid from the second highest floor, but each cuboid $\mathcal{D}'$ in concern already shares exactly one its upper facet $\partial_l(\mathcal{D}')$ with the ones from the top floor. So it reduces to the case of the previous paragraph. Then we similarly treat the third highest floor and so on until we reach the bottom floor.
\end{proof}
\end{proof}

Before presenting a concluding result of the section, we introduce the following notion. Let $\pi:(G,u)\rightarrow (H,v)$ be a surjective $\sigma$-homomorphism of unital Dedekind monotone $\sigma$-complete po-groups. A mapping $K:\mathbb R^n\to \Gamma^e_a(G,u)$ is said to be an {\it almost $n$-dimensional spectral resolution} if it satisfies (\ref{eq:(3.3)})--(\ref{eq:(3.5)}) and there is an element $u_0:=\bigvee_{(t_1,\ldots,t_n)} K(t_1,\ldots,t_n)\in \Gamma^e_a(G,u)$ such that $\pi(u_0)=v$.

Now, we present the main result of the section -- lifting of spectral resolutions.

\begin{theorem}\label{th:lifting}{\rm [}Lifting of Spectral Resolutions{\rm ]}
Let $\pi:(G,u)\rightarrow (H,v)$ be a $\sigma$-homomorphism of unital Dedekind monotone $\sigma$-complete po-groups and let $\pi$ satisfy (LP). Then each $n$-dimensional spectral resolution $F:\mathbb{R}^n\rightarrow H$ can be lifted to an almost $n$-dimensional spectral resolution $K: \mathbb{R}^n\rightarrow G$ such that  $\pi\circ K=F$.
\end{theorem}

\begin{proof}
According to Proposition~\ref{pr:main}, there are a countable subset $D \subset \mathbb R^n$ dense in $\mathbb R^n$ and a partial lift $L:D\rightarrow G$. Recall that $D=D_{\pi}^n$, where $D_{\pi}=\{k/2^l\colon l\ge 1, k\in \mathbb Z\}$ is a dense subset of $\mathbb{R}$, and $L$ is monotone on $D$.
Define
$K_0:\mathbb{R}^n\rightarrow G$ by prescription
\begin{equation}
K_0(t_1,\ldots,t_n)=\bigvee_{s_i\in D_\pi, s_i<t_i} L(s_1,\ldots,s_n).
\end{equation}

Note that the supremum exists because if we take two sequences $\{(s^j_1,\ldots,s_n^j)\}_j$ and $\{(u^j_1,\ldots,u_n^j)\}_j$ of elements of $D$ such that $(s^j_1,\ldots,s_n^j),(u^j_1,\ldots,u_n^j)\ll (t_1,\ldots,t_n)\in \mathbb R^n$ for each $j\ge 1$ and $\{(s^j_1,\ldots,s_n^j)\}_j \nearrow (t_1,\ldots,t_n)$ and $\{(u^j_1,\ldots,u_n^j)\}_j\nearrow (t_1,\ldots,t_n)$, then monotonicity of $L$ implies that
$$
\bigvee_j L(s^j_1,\ldots,s_n^j)\quad \text{and}\quad \bigvee_j L(u^j_1,\ldots,u_n^j)
$$
exist in $G$ and
$$
\bigvee_j L(s^j_1,\ldots,s_n^j)=\bigvee_j L(u^j_1,\ldots,u_n^j).
$$
Hence, $K_0(t_1,\ldots,t_n)$ is correctly defined and
$$K_0(t_1,\ldots,t_n)=\bigvee_j L(s^j_1,\ldots,s_n^j)=\bigvee_{(t_1,\ldots,t_n)\gg(s_1,\ldots,s_n)\in D}L(s_1,\ldots,s_n).
$$


Since each vertex of any cuboid is an $n$-tuple of reals, we will denote it by a Greek letter $\mathbf{\alpha}=(\alpha_1,\ldots,\alpha_n),\mathbf{\beta}= (\beta_1,\ldots,\beta_n)$, etc., and we write for simplicity $K_0(\mathbf{\alpha}):=K_0(\alpha_1,\ldots,\alpha_n)$.

Now, $K_0$ is monotone in each component (directly from definition) and all volume conditions hold: Given any $d$-cuboid in $\mathcal{C}$, we can write the volume condition in the form
\begin{equation}\label{eq:volumeIneq}
K_0(\alpha_1)+\cdots+K_0(\alpha_{2^{d-1}})\leq
K_0(\beta_1)+\cdots+K_0(\beta_{2^{d-1}}),
\end{equation}
where $\alpha_i$'s are all the vertices in $\mathcal{C}$ (hence $n$-tuples of reals) with even order and $\beta_i$'s are all the vertices in $\mathcal{C}$ with odd order. By the construction of $D$, for each $\epsilon >0$, there exists a $d$-cuboid $\mathcal{D}\subset D$, which is sufficiently close to $\mathcal{C}$: For each $i=1,\ldots,n$, we find $a^{\mathcal{D}}_i,b^{\mathcal{D}}_i\in D_{\pi}$, such that
$a^{\mathcal{C}}_i-\epsilon< a^\mathcal{D}_i<a^{\mathcal{C}}_i$ and
$b^{\mathcal{C}}_i-\epsilon< b^\mathcal{D}_i<b^{\mathcal{C}}_i$ (and hence the vertices of the cuboid $\mathcal{D}$ given by $a^{\mathcal{D}}_i$'s and
$a^{\mathcal{D}}_i$'s belong to $D$). Moreover, we can assume
$a^{\mathcal{D}}_i=b^{\mathcal{D}}_i$ for each $i\notin \Dim(\mathcal{C})$.

Given any $\alpha'_1,\ldots,\alpha'_{2^{d-1}}\in D$ such that for each $i=1,\ldots,n$, $\alpha'_i \ll\alpha_i$,
there is
$\epsilon>0$, for which $\alpha'_i\ll \alpha_i-\vec{\epsilon}$ for each $\alpha_i$, where $\vec{\epsilon}=(\epsilon,\ldots,\epsilon)\in \mathbb R^n$.
Hence, as we have proved above, there is a $d$-cuboid $\mathcal{D}$, such that for
each $i=1,\ldots, 2^{d-1}$ we have $\alpha'_i<\gamma_i<\alpha_i$, where $\gamma_i$
is the corresponding vertex of $\mathcal{D}$. From the volume condition for
$\mathcal{D}$ we deduce
$$ K_0(\alpha'_1)+\cdots+K_0(\alpha'_{2^{d-1}})\leq
K_0(\gamma_1)+\cdots+K_0(\gamma_{2^{d-1}})\leq
K_0(\delta_1)+\cdots+K_0(\delta_{2^{d-1}})\leq
K_0(\beta_1)+\cdots+K_0(\beta_{2^{d-1}}),$$
where $\delta_i$'s are all the vertices in $\mathcal{D}$ with even order.
Consequently, inequality~\eqref{eq:volumeIneq} holds by the definition of $K_0$ and
the fact that $+$ distributes over $\vee$.

The function $K_0$ yet does not have to vanish when some coordinate goes to $-\infty$. We
have to repair this. Define $K_0(-\infty,t_2,\ldots,t_n):=\bigwedge_t
K_0(t,t_2,\ldots,t_{n})$ (it exists since $K_0$ is monotone) and for $K_1(t_1,\ldots,t_n):=K_0(t_1,\ldots,t_n)-K_0(-\infty,t_2,\ldots,t_n)$, we have $\bigwedge_{t_1}K_1(t_1,\ldots,t_n)=0$.  We have to verify that
$K_1:\mathbb{R}^d\rightarrow G$ still satisfies the volume conditions. Let
$\mathcal{C}$ be a cuboid; two cases could occur (1) $1\in\Dim(\mathcal{C})$ or (2)
$1\not\in\Dim(\mathcal{C})$. In the case (1) we realize that
$|\mathcal{C}|_{K_1}=|\mathcal{C}|_{K_0}$, since for any edge $e$ with
$\Dim(e)=\{1\}$ clearly $|e|_{K_0}=|e|_{K_1}$ and $\mathcal{C}$ could be as element of $\mathcal{A}$ written as
a linear combination of such edges (since $1\in\Dim(\mathcal{C})$). Now consider the
case (2). For each $i\in\mathbb{N}_0$ define $\mathcal{C}_n$ a cuboid which shares
with $\mathcal{C}$ all coordinates except $a_1^{\mathcal{C}_n}=a_1^{\mathcal{C}}-n$
(hence it arises by moving the original cuboid $\mathcal{C}$ down in the direction
of the first axis by $n$). See that for each $i\in\mathbb{N}_0$ we have
$\mathcal{C}_{i}-\mathcal{C}_{i+1}$  equals a cuboid $\mathcal{D}_i$ (as an element
of $\mathcal{A}$) for which $1\in\Dim(\mathcal{D}_i)$. So by the case (1),
$|\mathcal{C}_i|_{K_1}=|\mathcal{C}_{i+1}|_{K_1}+|\mathcal{D}_i|_{K_1}\geq
|\mathcal{C}_{i+1}|_{K_1}$, that is $\{|\mathcal{C}_i|_{K_1}\}_{i\in\mathbb{N}_0}$ is
a decreasing sequence. By monotone $\sigma$-completeness, there is $c=\bigwedge_i
|\mathcal{C}_i|_{K_1}$, for which we want to prove $c=0$, as this trivially implies
$|\mathcal C|_{K_1}=|\mathcal C_0|_{K_1}\geq 0$. Let us list $\alpha^j_i$ ($\beta^j_i$, respectively),
$j=1,\ldots, 2^{\dim(\mathcal{C})-1}$ all the vertices in $\mathcal{C}_i$ with even
(odd, respectively) order. Note that as $i\rightarrow\infty$, $K_1(\alpha^j_i)\searrow 0$
as well as $K_1(\beta^j_i)\searrow 0$ for each $j=1,\ldots,2^{\dim(\mathcal{C})}$
(since the first coordinate of $\alpha_i^j$ ($\beta_i^j$, resp.) goes to $-\infty$).
We have the following estimations:
$$-(K_1(\alpha_i^1)+\cdots+K_1(\alpha_i^{2^(d-1)}))\leq
K_1(\beta_i^1)+\cdots+K_1(\beta_i^{2^{(d-1)}})-(K_1(\alpha_i^1)+\cdots+K_1(\alpha_i^{2^{(d-1)}}))=|\mathcal{C}_i|_{K_1},$$
$$|\mathcal{C}_i|_{K_1} =
K_1(\beta_i^1)+\cdots+K_1(\beta_i^{2^{(d-1)}})-(K_1(\alpha_i^1)+\cdots+K_1(\alpha_i^{2^{(d-1)}}))\leq
K_1(\beta_i^1)+\cdots+K_1(\beta_i^{2^{(d-1)}}).$$
As $-(K_1(\alpha_i^1)+\cdots+K_1(\alpha_i^{2^{(d-1)}}))\nearrow 0$ and
$K_1(\beta_i^1)+\cdots+K_1(\beta_i^{2^{(d-1)}})\searrow 0$ we have $c=0$.

Suppose $\alpha_i\nearrow \alpha:=(t_1,\ldots,t_n)$, then $K_1(\alpha)\geq
K_1(\alpha_i)\geq K_0(\alpha_i)-K_0(-\infty,t_2,\ldots,t_n)\nearrow
K_0(\alpha)-K_0(-\infty,t_2,\ldots,t_n) = K_1(\alpha)$.
If $K_0(t_1,\ldots,t_{i-1},-\infty,t_i,\ldots,t_n)=0$ for some $i$, clearly
$$
K_1(t_1,\ldots,t_{i-1},-\infty,t_{i+1},\ldots,t_n)=0
$$
as well. Consequently, $K_1$ satisfies all volume conditions, consequently, $K_1$ is monotone.

If $K_1(t_1,-\infty,t_3,\ldots,t_n):=\bigwedge_{t_2}K_1(t_1,\ldots,t_n)>0$, we repeat the above procedure with $K_2(t_1,\ldots,t_n)= K_1(t_1,\ldots,t_n)-K_1(t_1,-\infty,t_2,\ldots,t_n)$, $(t_1,\ldots,t_n)\in \mathbb R^n$. Then $\bigwedge_{t_1} K_2(t_1,\ldots,t_n)=0=\bigwedge_{t_2} K_2(t_1,\ldots,t_n)$. In particular, all the volume conditions hold and (specially) $K_2$ is monotone.
By induction, for each $i=1,\ldots,n-1$, we define
$$
K_i(t_1,\ldots,t_i,-\infty,t_{i+2},\ldots,t_n) = \bigwedge _{t_{i+1}} K_i(t_1,\ldots,t_n)
$$
and
$K_{i+1}(t_1,\ldots,t_n)=K_i(t_1,\ldots,t_n)- K_i(t_1,\ldots,t_i,-\infty,t_{i+2},\ldots,t_n)$. Then $\bigwedge_{t_j}K_{i+1}(t_1,\ldots,t_n)=0$ for $j=1,\ldots,i+1$.

Finally, $\pi\circ K_1= F$, since $\pi(K_0(-\infty,t_2,\ldots,t_n))=0$ for each real $t_i$'s, because
$\pi$ is a $\sigma$-homomorphism. If we repeat the process for all coordinates, we obtain a mapping $K:=K_n:\mathbb{R}^n\rightarrow G$ which is monotone and satisfies conditions~(\ref{eq:(3.3)})--(\ref{eq:(3.5)}). In addition, $\pi\circ K=F$. If we define $u_0=\bigvee_{(s_1,\ldots,s_n)}K(s_1,\ldots,s_n)$ (it exists due to monotonicity of $K$), then $\pi(u_0)=v$, and we see that $K$ is an almost $n$-dimensional spectral resolution on $G$ which finishes the proof.
\end{proof}

\section{$n$-dimensional Spectral Resolutions and $n$-dimensional Observables}

The following two results can be proved by applying the Loomis--Sikorski Theorem for $\sigma$-complete MV-algebras as well as using Lifting Theorem \ref{th:lifting}. They are the second main results of the paper.

If $M$ is a $\sigma$-complete MV-algebra or a monotone $\sigma$-complete effect algebra with (RDP), then it is always an interval in a unital $\sigma$-complete $\ell$-group or in a unital Dedekind monotone $\sigma$-complete po-group with interpolation. So an $n$-dimensional spectral resolution is a mapping $F: \mathbb R^n\to M$ such that (\ref{eq:(3.1)})--(\ref{eq:(3.5)}) holds, where $u=1$ is the top element of $M$.

For the Loomis--Sikorski Theorem, see \cite{Dvu1,Mun1}, we need the notion of a tribe of fuzzy sets: A {\it tribe} on $\Omega \ne \emptyset$
is a family ${\mathcal T}$ of fuzzy sets from $[0,1]^\Omega$ such
that (i) $1 \in {\mathcal T}$, (ii) if $f \in {\mathcal T}$, then $1 - f \in
{\mathcal T}$, and (iii) if $\{f_n\}_n$ is a sequence from ${\mathcal T}$,
then $\min \{\sum_{n=1}^\infty f_n,1 \}\in {\mathcal T}$.  A tribe is
always a $\sigma$-complete MV-algebra where all operations are defined by points.

\begin{theorem}\label{th:4.1}
Let $M$ be a $\sigma$-complete MV-algebra. Then there is a one-to-one  correspondence between $n$-dimensional spectral resolutions and $n$-dimensional observables on $M$.
\end{theorem}

\begin{proof} The first proof.
Let $\Omega=\partial \mathcal S(M)$. Then $\Omega$ is a basically disconnected Hausdorff compact topological space under the weak topology of extremal states. Given $a\in M$, let $\hat a:\partial \mathcal S(M)\to [0,1]$ be a mapping defined by $\hat a(s):=s(a)$, $s \in \partial \mathcal S(M)$. Then $\hat a$ is a continuous function on $\partial \mathcal S(M)$, $\hat a \in \Gamma(C(\Omega),1_\Omega))$, and $(C(\Omega),1_\Omega)$ is a Dedekind $\sigma$-complete unital $\ell$-group, even a Riesz space.

The mapping $\phi: a\to \hat a$ is an injective MV-homomorphism from $M$ into $\Gamma(C(\Omega),1_\Omega)$. Let $\mathcal T$ be a tribe of fuzzy sets generated by $\{\hat a\colon a \in M\}$. For $f,g: [0,1]^\Omega$, we write $f\sim g$ if $\{s\in \partial \mathcal S(M)\colon f(s)\ne g(s)\}$ is a meager set. According to the proof of the Loomis--Sikorski Theorem, \cite[Thm 4.1]{Dvu1}, $\mathcal T=\{f\in [0,1]^\Omega\colon f \sim \hat a$ for some $a\in M\}$, and the mapping $\pi:\mathcal T\to M$, defined by $\pi(f)=a$ iff $f\sim \hat a$, is a surjective $\sigma$-homomorphism of MV-algebras.

Hence, let $F:\mathbb R^n\to M$ be an $n$-dimensional spectral resolution. We define a mapping $K_0:\mathbb R^n \to \mathcal T$ by $K_0(t_1,\ldots,t_n)=\phi(F(t_1,\ldots,t_n))$, $t_1,\ldots,t_n\in \mathbb R$. Then $K_0$ is monotone, satisfies the volume condition, and $\pi\circ K_0=F$. In the following, we will calculate $\bigvee$ and $\bigwedge$ in $\mathcal T$.

We define
$$
K_0(-\infty,t_2,\ldots,t_n)=\bigwedge_{t_1}K_0(t_1,\ldots,t_n),
$$
then $K_0(-\infty,t_2,\ldots,t_n)\in \mathcal T$ (since $K$ is monotone, it is enough to take $t_1\in \mathbb Q$, $t_1\searrow -\infty$),
and let $K_1(t_1,\ldots,t_n)=K_0(t_1,\ldots,t_n)-   K_0(-\infty,t_2,\ldots,t_n)$, $t_1,\ldots,t_n\in \mathbb R$. Then $\bigwedge_{t_1} K_1(t_1,\ldots,t_n)=0$ and $\pi(K_1(t_1,\ldots,t_n))=F(t_1,\ldots,t_n)$.

Now, we show that $K_1$ is monotone and it satisfies all the volume condition of the form (\ref{eq:(3.6)}). Clearly that $K_1(s_1,t_2,\ldots,t_n)\le K_1(t_1,\ldots,t_n)$ where $s_1\le t_1$.

Let $(s_1,s_2,t_3,\ldots,t_n)\le (t_1,t_2,\ldots,t_n)$. Using (\ref{eq:(3.7)}) for $K_0$, we have
\begin{align*}
K_0(t_1,t_2,t_3,\ldots,t_n)+ K_0(s_1,s_2,t_3,\ldots,t_n)\ge
K_0(t_1,s_2,t_3,\ldots,t_n) +K_0(s_1,t_2,t_3,\ldots,t_n).
\end{align*}
If $s_1\searrow -\infty$, we have
\begin{align*}
K_0(t_1,t_2,t_3,\ldots,t_n)+ K_0(-\infty,s_2,t_3,\ldots,t_n)&\ge
K_0(t_1,s_2,t_3,\ldots,t_n)+K_0(-\infty,t_2,t_3,\ldots,t_n)\\
K_0(t_1,t_2,t_3,\ldots,t_n)-K_0(-\infty,t_2,t_3,\ldots,t_n)&\ge
K_0(t_1,s_2,t_3,\ldots,t_n) - K_0(-\infty,s_2,t_3,\ldots,t_n)\\
K_1(t_1,t_2,t_3,\ldots,t_n)&\ge K_1(t_1,s_2,t_3,\ldots,t_n).
\end{align*}

The same is true for each $i=3,\ldots,n$, i.e.
$$
K_1(t_1,\ldots,t_n)\ge K_1(s_1,t_2,\ldots,t_{i-1},s_i,t_{i+1},\ldots,t_n)
$$
whenever $s_i \le t_i$. Hence, if $(s_1,\ldots,s_n)\le (t_1,\ldots,t_n)$, then $K_1(s_1,\ldots,s_n)\le K_1(t_1,s_2,\ldots,s_n)\le K_1(t_1,t_2,s_3,\ldots,s_n)\le \cdots \le K_1(t_1,\ldots,t_i,s_{i+1},\ldots,s_n)\le \cdots \le K_1(t_1,\ldots,t_n)$.

Using the left-hand side of (\ref{eq:(3.5)}) for $K_1$, we see that it is equal to the left-hand side of (\ref{eq:(3.6)}) for $K_0\ge 0$, so that the volume condition (\ref{eq:(3.5)}) for $K_1$ also holds.

Set
$$
  K_1(t_1,-\infty,t_3,\ldots,t_n)=\bigwedge_{t_2}K_1(t_1,\ldots,t_n)
$$
and $K_2(t_1,\ldots,t_n)=K_1(t_1,\ldots,t_n) -   K_1(t_1,-\infty,t_3,\ldots,t_n)$. Then $\bigwedge_{t_2}K_2(t_1,\ldots,t_n)=0$ and $\pi\circ K_2=F$. Moreover,
\begin{align*}
K_2(-\infty,t_2,\ldots,t_n)&=\bigwedge_{t_1} K_2(t_1,\ldots,t_n)= \bigwedge_{t_1}(K_1(t_1,\ldots,t_n) -K_2(t_1,-\infty,t_3,\ldots,t_n))\\
&\le \bigwedge_{t_1} K_1(t_1,\ldots,t_n) = 0.
\end{align*}
In an analogous way, we can establish that $K_2$ is monotone and for it every volume condition (\ref{eq:(3.5)}) holds.

By induction, for each $i=1,\ldots,n-1$, we define
$$
  K_{i}(t_1,\ldots,t_{i},-\infty,t_{i+2},\ldots,t_n)=\bigwedge_{t_{i+1}} K_i(t_1,\ldots,t_n)
$$
and $K_{i+1}(t_1,\ldots,t_n)= K_{i}(t_1,\ldots,t_n)-   K_{i}(t_1,\ldots,t_{i},-\infty,t_{i+2},\ldots,t_n)$.
Then $\bigwedge_{t_j}K_{i+1}(t_1,\ldots,t_n)=0$ for each $j=1,\ldots,i+1$ and $\pi\circ K_{i+1}=F$. After finitely many steps, we define also $K_n$. Then each $K_1,\ldots,K_n$ is monotone and each $K_i$ satisfies all the volume conditions (\ref{eq:(3.6)}).

Define a mapping $K: \mathbb R^n \to \mathcal T$ by
\begin{equation}\label{eq:4.1}
K(t_1,\ldots,t_n)= \bigvee_{(s_1,\ldots,s_n)\ll(t_1,\ldots,t_n)} K_n(s_1,\ldots,s_n),\quad t_1,\ldots,t_n \in \mathbb R.
\end{equation}

First we have to note that $K$ is correctly defined, because due to monotonicity of $K_n$, it is enough to use $s_1,\ldots,s_n\in \mathbb Q$. Therefore, $\pi \circ K= F$.
Clearly, $K$ is monotone. We assert that $K$ is a mapping satisfying  (\ref{eq:(3.3)})--(\ref{eq:(3.5)}).

Equality (\ref{eq:(3.3)}): Clearly $\bigvee_{(s_1,\ldots,s_n)\ll(t_1,\ldots,t_n)} K(s_1,\ldots,s_n)\le K(t_1,
\ldots,t_n)$. To prove the opposite inequality, let $(s_1'',\ldots,s_n'')\ll(t_1,\ldots,t_n)$. There is an $n$-tuple $(s_1',\ldots,s_n')$ such that $(s_1'',\ldots,s_n'')\ll (s_1',\ldots,s_n')\ll (t_1,\ldots,t_n)$. Hence,
$$
K_n(s_1'',\ldots,s_n'')\le K(s_1',\ldots,s_n') \le K(t_1,\ldots,t_n)
$$
taking supremum over all $(s_1'',\ldots,s_n'')\ll(t_1,\ldots,t_n)$, we get
$$
K(t_1,\ldots,t_n)=\bigvee_{(s''_1,\ldots,s''_n)\ll (t_1,\ldots,t_n)} K_n(s_1'',\ldots,s_n'')\le\bigvee_{(s_1',\ldots,s_n')\ll(t_1,\ldots,t_n)} K(s_1',\ldots,s_n') \le K(t_1,\ldots,t_n).
$$

Equality (\ref{eq:(3.4)}) follows from:
$$
\bigwedge_{t_i}K(t_1,\ldots,t_n)=\bigwedge_{t_i} \bigvee_{(s_1,\ldots,s_n)\ll(t_1,\ldots,t_n)} K_n(s_1,\ldots,s_n) \le \bigwedge_{t_i}K_n(t_1,\ldots,t_n)=0.
$$
Now, we verify the volume condition (\ref{eq:(3.5)}) for $K$. Let $A=\langle a_1,b_1)\times\cdots \times \langle a_n,b_n)$ be given.
We note that $K_0$ satisfies the volume condition. Therefore, we express the volume condition in the form $LHS_{K_0}(A) \ge RHS_{K_0}(A)$, where on both sides there are sums of positive terms of $K_0$. Subtracting from both side the elements $K_0(-\infty,t_2,\ldots,t_n)$, we obtain the volume condition for $K_1$. Analogously, we can show that each $K_2,\ldots, K_n$ satisfies the volume condition.

For any integer $k\ge 1$, let $A_k=\langle a_1-1/k,b_1-1/k)\times\cdots \times \langle a_n-1/k,b_n-1/k)$.
We denote by $LHS_{K_n}(A_k)$, $RHS_{K_n}(A_k)$ the left-hand and right-hand side of the volume condition of $A_k$ in $K_n$ and let $LHS_K(A)$, $RHS_K(A)$ be analogous expressions for $A$ in the mapping $K$.
Since $K_n$ satisfies the volume condition, for each $k\ge 1$, we have
\begin{align*}
LHS_{K_n}(A_k) &\ge RHS_{K_n}(A_k)\\
LHS_K(A)=\lim_k LHS_{K_n}(A_k) &\ge \lim_k RHS_{K_n}(A_k)= RHS_K(A),
\end{align*}
so that $K$ satisfies the volume condition.
Denote by $u_0=\bigvee_{(t_1,\ldots,t_n)}K(t_1,\ldots,t_n)$. Then $\pi(u_0)=1$.

Therefore, we have a system $\{K(t_1,\ldots,t_n)\colon t_1,\ldots,t_n\in \mathbb R\}$ of fuzzy sets on $\Omega$ in the tribe $\mathcal T$. For each fixed $\omega \in \Omega$, the function $K_\omega: \mathbb R^n \to [0,1]$ defined by $K_\omega(t_1,\ldots,t_n):=K(t_1,\ldots,t_n)(\omega)$, $t_1,\ldots,t_n \in \mathbb R$, is left continuous, going to $0$ if $t_i \to -\infty$ with non-negative increments. According to \cite[Thm 2.25]{Kal}, there is a unique $\sigma$-additive finite measure $P_\omega$ on $\mathcal B(\mathbb R^n)$ such that $P_\omega((-\infty,t_1)\times \cdots\times (-\infty,t_n))=F_\omega(t_1,\ldots,t_n)$.  Therefore, we have a mapping $\xi:\mathcal B(\mathbb R^n)\to [0,1]^\Omega$ such that $\xi(A)(\omega)= P_\omega(A)$ for all $A \in \mathcal B(\mathbb R^n)$ and all $\omega \in \Omega$.  We denote by $\mathcal K=\{A\in \mathcal B(\mathbb R^n)\colon \xi(A)\in \mathcal T\}$. Then $\mathcal K$ contains $\mathbb R^n$, all intervals of the form $(-\infty,t_1)\times\cdots\times (-\infty,t_n)$, and is closed under complements and unions of disjoint sequences, i.e. $\mathcal K$ is a Dynkin system and by the Sierpi\'nski Theorem, \cite[Thm 1.1]{Kal}, $\mathcal K = \mathcal B(\mathbb R^n)$.
Then $x(A):=\pi(\xi(A))$, $A \in \mathcal B(\mathbb R^n)$, is an $n$-dimensional observable on $M$ such that $x((-\infty,t_1)\times\cdots \times (-\infty,t_n))=F(t_1,\ldots,t_n)$, $t_1,\ldots,t_n \in \mathbb R$.

Uniqueness of $x$: Let $y$ be another $n$-dimensional observable such that $y((-\infty,t_1)\times\cdots \times (-\infty,t_n))=F(t_1,\ldots,t_n)$, $t_1,\ldots,t_n \in \mathbb R$. Let $\mathcal H=\{A \in \mathcal B(\mathbb R^n)\colon x(A)=y(A)\}$. Then $\mathcal H$ is a Dynkin system containing all intervals $(-\infty,t_1)\times \cdots\times (-\infty, t_n)$ so that by the Sierpi\'nski Theorem, $\mathcal H= \mathcal B(\mathbb R^n)$ which shows $x=y$.
\end{proof}

Now, we present the second proof using Lifting Theorem \ref{th:lifting}:

\begin{proof}
According to the Loomis--Sikorski theorem for $\sigma$-complete MV-algebras, see \cite{Dvu1,Mun1}, there is a tribe $\mathcal T$ of fuzzy sets in $\Omega\ne \emptyset$ and a surjective $\sigma$-homomorphism $\pi:\mathcal T \to M$.

We assert that $\pi$ has the lifting property. Let $L,U \subseteq \mathcal{T}$ be finite and $w\in M$ such that $L\leq U$ and $\pi(L)\leq w\leq \pi(U)$. Put $u=\bigvee U$ and $v=\bigwedge V$. Then $u\le v$ and for $w\in M$ there is $f_1\in \mathcal T$ such that $\pi(f_1)=w$. If we set $f=u\vee(v\wedge f_1)$, then $L\le f \le U$ and $\pi(L)\le \pi(f)=w\le \pi(U)$.

The $\sigma$-complete MV-algebra $M$ is an interval in a Dedekind $\sigma$-complete unital $\ell$-group $(H,v)$, i.e. we can assume $M=\Gamma(H,v)$ and similarly, the tribe $\mathcal T$ is an interval in the Dedekind $\sigma$-complete $\ell$-group $\mathcal G$ of bounded functions on $\Omega$ with the pointwise ordering, so that $\mathcal T=\Gamma(\mathcal G,1_\Omega)$. In addition, the surjective $\sigma$-homomorphism $\pi$ can be extended to a unique surjective $\sigma$-homomorphism from $\mathcal G$ onto $H$.

Applying Lifting Theorem \ref{th:lifting}, there is an almost $n$-dimensional spectral resolution $K:\mathbb R^n\to \mathcal T$ such that $\pi\circ K=F$. Therefore, we have a system $\{K(t_1,\ldots,t_n)\colon t_1,\ldots,t_n \in \mathbb R\}$ of fuzzy sets from the tribe $\mathcal T$. To finish the proof, we use literally the rest of the first proof of the present theorem.
\end{proof}

In the next result, we extend the ideas of the proof of the latter theorem for monotone $\sigma$-complete effect algebras with (RDP). Instead of tribe, we need the following notion: An {\it effect-tribe}  is any system ${\mathcal T}$ of fuzzy sets on $\Omega\ne \emptyset $ such that (i) $1 \in {\mathcal T}$, (ii) if $f\in {\mathcal T},$ then $1-f \in {\mathcal T}$, (iii) if $f,g \in {\mathcal T}$,
$f \le 1-g$, then $f+g \in {\mathcal T},$ and (iv) for any sequence
$\{f_n\}_n$ of elements of ${\mathcal T}$ such that $\{f_n\}_n \nearrow f$
(pointwisely), then $f \in {\mathcal T}$. Then every
effect-tribe is a monotone $\sigma$-complete effect algebra where all operations are defined by points.

Let $f$ be a real-valued function on ${\mathcal S}(E)$, where $E$ is not stateless.  We define
$$
N(f) :=\{s \in \partial{\mathcal S}(E) :\, f(s) \ne 0\}.
$$

\begin{theorem}\label{th:4.2}
Let $E$ be a monotone $\sigma$-complete effect algebra satisfying the Riesz Decomposition Property. Then there is a one-to-one  correspondence between $n$-dimensional spectral resolutions and $n$-dimensional observables on $E$.
\end{theorem}

\begin{proof}
Given $a \in E$, let $\hat a:\mathcal S(E)\to [0,1]$ be a function defined by $\hat a(s)=s(a)$, $s \in \mathcal S(E)$. Then $\hat a$ is an affine continuous function on $\mathcal S(E)$. The mapping $\varphi: E \to   E=\{\hat a\colon a \in E\}$, defined by $\varphi(a)=\hat a$, $a \in E$, is an isomorphism of effect algebras $E$ and $\widehat E$. If we define $\mathcal T$ as the effect-tribe generated by $\widehat E$, then due to the proof of the Loomis--Sikorski Theorem, see \cite{BCD}, $\mathcal T$ is the set of all fuzzy sets $f$ on $\Omega=\mathcal S(E)$ such that there is $a \in E$ with $N(f-\hat a)$ being a meager set. Then $\mathcal T$ is an effect-tribe with (RDP) and the mapping $\pi:\mathcal T \to E$, given by $\pi(f)=a$ iff $N(f-\hat a)$ is a meager set, is a $\sigma$-epimorphism of effect algebras.

Let $F$ be an $n$-dimensional spectral resolution on $E$. Analogously as in the proof of Theorem \ref{th:4.1}, we define a mapping $K_0:\mathcal B(\mathbb R^n)\to \mathcal T$ by $K_0(t_1,\ldots,t_n)= \varphi(F(t_1,\ldots,t_n))$, $t_1,\ldots,t_n \in \mathbb R$. Then $K_0$ is monotone, it satisfies the volume condition, and $\pi\circ K_0=F$.

We define
$$
  K_0(-\infty,t_2,\ldots,t_n)=\bigwedge_{t_1}K_0(t_1,\ldots,t_n).
$$
Using monotonicity of $K_0$, we can see that the element $K_0(-\infty,t_2,\ldots,t_n)$ exists in $\mathcal T$.
If we set
$$
K_1(t_1,\ldots,t_n)=K_0(t_1,\ldots,t_n)-   K_0(-\infty,t_2,\ldots,t_n), \quad t_1,\ldots,t_n\in \mathbb R,
$$ then $\bigwedge_{t_1} K_1(t_1,\ldots,t_n)=0$.

Having $K_i$ for $i=1,\ldots,n-1$, we construct $K_{i+1}$ by induction as follows: Let
$$
  K_{i}(t_1,\ldots,t_{i},-\infty,t_{i+2},\ldots,t_n) =\bigwedge_{t_{i+1}} K_i(t_1,\ldots,t_n)
$$
and $K_{i+1}(t_1,\ldots,t_n)= K_{i}(t_1,\ldots,t_n)-   K_{i}(t_1,\ldots,t_{i},-\infty,t_{i+2},\ldots,t_n)$.
Then $\bigwedge_{t_j}K_{i+1}(t_1,\ldots,t_n)=0$ for each $j=1,\ldots,i+1$ and $\pi\circ K_{i+1}=F$. Moreover, due to the same argumentation as in the proof of Theorem \ref{th:4.1}, each $K_1,\ldots,K_n$ is monotone and each $K_i$ satisfies every volume condition of the form (\ref{eq:(3.5)}).

Finally, we define a mapping $K: \mathbb R^n \to \mathcal T$ by
$$
K(t_1,\ldots,t_n)= \bigvee_{(s_1,\ldots,s_n)\ll (t_1,\ldots,t_n)} K_n(s_1,\ldots,s_n),\quad t_1,\ldots,t_n \in \mathbb R.
$$
We stress that each $K(t_1,\ldots,t_n)$ is defined in $\mathcal T$.

Repeating the proof of Theorem \ref{th:4.1}, we can show that $K$ satisfies the volume condition as well as (\ref{eq:(3.3)})--(\ref{eq:(3.4)}).

If we define $u_0=\bigvee_{(t_1,\ldots,t_n)}K(t_1,\ldots,t_n)$, then $u_0 \in \mathcal T$ and $\pi(u_0)=1$.

To finish the proof, we follow literally the rest of the proof of Theorem \ref{th:4.1}, and it gives a unique $n$-dimensional observable $x$ on $E$ such that $x((-\infty,t_1)\times\cdots\times (-\infty,t_n))=F(t_1,\ldots,t_n)$ for all $t_1,\ldots,t_n \in \mathbb R$.
\end{proof}

We present also another proof using Lifting Theorem \ref{th:lifting}.

\begin{proof}
Take the effect-tribe $\mathcal T$ of fuzzy sets on $\Omega =\mathcal S(E)$ from the first proof. Due to the Loomis--Sikorski Theorem, \cite{BCD}, $\mathcal T$ satisfies (RDP) and the mapping $\pi: \mathcal T\to E$ is a surjective homomorphism preserving monotone sequences from $\mathcal T$.

Since $E$ satisfies (RDP), there is a unital monotone $\sigma$-complete po-group $(H,v)$ with interpolation such that $E=\Gamma^e_a(H,v)$ and similarly, there is a unital monotone $\sigma$-complete po-group with interpolation $(\mathcal G,1_\Omega)$ of bounded functions on $\Omega$ the pointwise ordering such that $\mathcal T =\Gamma^e_a(\mathcal G,1_\Omega)$. In addition, $\pi$ can be extended to a unique surjective homomorphism from $\mathcal G$ onto $H$ preserving all bounded monotone sequences from $\mathcal G$.

Now, we show that $\pi$ has (LP). Due to Claim from the proof of \cite[Thm 4.2]{DvLa1}, we have: If $f \sim a$, $g\sim b$ for $f,g \in \mathcal T$, $a,b \in E$, and if $a\wedge b$ $(a\vee b)$ exists in $E$, then $\min\{f,g\}\in \mathcal T$ $(\max\{f,g\}\in \mathcal T)$ and $\min\{f,g\} \sim a\wedge b$ $(\max\{f,g\}\sim a\vee b)$.

This property implies that $\pi$ has (LP).

Now, we can apply Lifting Theorem \ref{th:lifting}, so that there is an almost $n$-dimensional spectral resolution $K:\mathbb R^n\to \mathcal T$ such that $\pi\circ K=F$. Therefore, we have a system $\{K(t_1,\ldots,t_n)\colon t_1,\ldots,t_n \in \mathbb R\}$ of fuzzy sets from the effect-tribe $\mathcal T$. The rest of the proof follows the same ideas as those at the end of the  proof of Theorem \ref{th:4.1}.
\end{proof}

Now, let us assume that $E$ is a $\sigma$-complete effect algebra. We note that such an example is not necessarily an interval effect algebra. Therefore, in the notion of the volume condition, we have take a small change to have a sense in (\ref{eq:(3.5)}).

Thus, we assume that we have a mapping $F:\mathbb R^n\to E$ such that equalities (\ref{eq:(3.1)})--(\ref{eq:(3.4)}) hold and instead of (\ref{eq:(3.2)}) we have on the right-hand side the element $1$ instead of $u$, i.e. $\bigvee_{(s_1,\ldots,s_n)}F(s_1,\ldots,s_n)=1$. Whence, let for each semi-closed block $A=\langle a_1,b_1)\times\cdots\times \langle a_n,b_n)$, for each permutation $(i_1,\ldots,i_n)$ of $(1,\ldots,n)$, and for all $t_1,\ldots,t_n\in \mathbb R$ we have
\begin{align*}
&\Delta_{i_1}(a_{i_1},b_{i_1})F(t_1,\ldots,t_n)\ge 0\\
& \Delta_{i_2}(a_{i_2},b_{i_2})\Delta_{i_1}(a_{i_1},b_{i_1})F(t_1,\ldots,t_n)= \Delta_{i_1}(a_{i_1},b_{i_1})\Delta_{i_2}(a_{i_2},b_{i_2})F(t_1,\ldots,t_n)\ge 0\\
&\vdots\\
&\Delta_{i_j}(a_{i_j},b_{i_j})\cdots \Delta_{i_1}(a_{i_1},b_{i_1})F(t_1,\ldots,t_n)= \Delta_{l_j}(a_{l_j},b_{l_j})\cdots \Delta_{l_1}(a_{l_1},b_{l_1})F(t_1,\ldots,t_n)\ge 0,\\
& \quad \quad  \quad 1\le j\le n,\ l_1,\ldots, l_j\in \{i_1,\ldots,i_j\}\\
&\vdots\\
& \Delta_{i_n}(a_{i_n},b_{i_n})\cdots \Delta_{i_2}(a_{i_2},b_{i_2})\Delta_{i_1}(a_{i_1},b_{i_1})F(t_1,\ldots,t_n) =\Delta_{1}(a_{1},b_{1})\cdots \Delta_{n}(a_{n},b_{n})F(t_1,\ldots,t_n)\ge 0,
\end{align*}
then $F$ is said to be an $n$-{\it dimensional spectral resolution} with values in a $\sigma$-complete effect algebra.
We note that if $F$ is an $n$-dimensional spectral resolution in the case of an interval effect algebra $\Gamma^e_a(G,u)$, see Definition \ref{de:3.1}, then due to Proposition \ref{pr:prop}, it is an $n$-dimensional spectral resolution also in a new sense.

Given a semi-closed block $A=\langle a_1,b_1)\times \cdots\times \langle a_n,b_n)$ and an $n$-dimensional spectral resolution $F$ on a $\sigma$-complete effect algebra $E$, we put
$$
V_F(A) =\Delta_{1}(a_{1},b_{1})\cdots \Delta_{n}(a_{n},b_{n})F(t_1,\ldots,t_n).
$$

In the same way as in Section 2, we can define, for each $i=1,\ldots,n$, the operator $\Delta_i(a,b)G(s_1,\ldots,s_n)= G(s_1,\ldots,s_{i-1},b,s_{i+1},\ldots,s_n)$, where $a=-\infty$ and $b \in \mathbb R$ and $G:\mathbb R^n\to E$. If $A=C_1\times\cdots\times C_n$, where each $C_i$ is either $C_i=(-\infty,b_i)$ or $C_i=\langle a_i,b_i)$ with reals $a_i\le b_i$, we define
$$
V_F(A)=\Delta_1(a_1,b_1)\cdots \Delta_n(a_n,b_n)F(s_1,\ldots,s_n).
$$
Due to (\ref{eq:(3.4)}), $V_F(A)\ge 0$.

\begin{lemma}\label{le:4.3}
Let $F$ be an $n$-dimensional spectral resolution on a $\sigma$-complete effect algebra $E$. Let $A=C_1\times\cdots\times C_n$, where each $C_i$ is either $C_i=(-\infty,b_i)$ or $C_i=\langle a_i,b_i)$, where $a_i\le b_i$ are real numbers.
Let for a unique $i=1,\ldots,n$, $C_i= C'_i\cup C''_i$, where $C'_i$ is either $C'_i=(-\infty,c_i)$ and $C''_i=\langle c_i,b_i)$ or $C'_i=\langle a_i,c_i)$ and $C''_i=\langle c_i,b_i)$, where $a_i\le c_i \le b_i$.
Define $A_1= C_1\times \cdots \times C'_i\times \cdots\times C_n$ and $A_2=C_1\times \cdots \times C''_i\times \cdots \times C_n$.
Then $A= A_1\cup A_2$, $A_1\cap A_2=\emptyset$, and
$V_F(A)=V_F(A_1)+V_F(A_2)$.

\end{lemma}

\begin{proof}
For simplicity, we assume that $i=1$. Define $G(t_1)=\Delta_2(a_2,b_2)\cdots\Delta_n(a_n,b_n)F(t_1,\ldots,t_n)$. Then $V_F(A)=\Delta_1(a_1,b_1)G(t_1)$. Due to properties of $F$ and effect algebras, we see that $V_F(A)=G(b_1)-G(a_1)\ge 0$  so that $V_F(A)=G(b_1)-G(a_1)=(G(b_1)-G(c_1))+(G(c_1)- G(a_1))= V_F(A_1)+V_F(A_2)$.
\end{proof}

Now we introduce the notion of compatibility of two elements $a$ and $b$ of an effect algebra. We say that two elements $a,b \in E$ are {\it compatible}, and we write $a \leftrightarrow b$, if there are three elements $a_1, b_1,c$ such that $a= a_1+c,$ $b = b_1 +c$ and $a_1+b_1+c$ is defined in $E.$ For example, every two elements of an MV-algebra are compatible; indeed, if we put $c=a\wedge b$, $a_1=a-c$, and $b_1=b-c$, then $a_1+b_1+c$ exists and $a\oplus b=a_1+b_1+c$.
A {\it block} is any maximal system of mutually compatible elements of $E$. If $a\le b$, then $a \leftrightarrow b$~:  $a = 0 + a$ and $b = (b-a)+a$. We note that if $x$ is an $n$-dimensional observable and $A,B \in \mathcal B(\mathbb R^n)$, then $x(A)\leftrightarrow x(B)$. Indeed, $x(A\cup B)=x(A\setminus B)+x(A\cap B) + x(B\setminus A)$ which establishes the compatibility of $x(A)$ and $x(B)$.

\begin{theorem}\label{th:4.4}
Let $F$ be an $n$-dimensional spectral resolution on a $\sigma$-complete effect algebra $E$. There is a unique $n$-dimensional observable $x$ on $E$ such that $x((-\infty,s_1,)\times\cdots\times (-\infty,s_n))=F(s_1,\ldots,s_n)$ for all $s_1,\ldots,s_n \in \mathbb R$.
\end{theorem}

\begin{proof}
Let $a=F(s_1,\ldots,s_n)$ and $b= F(t_1,\ldots,t_n)$.
By induction on $n$ we prove that, for every mapping $F:\mathbb R^n \to E$ satisfying (\ref{eq:(3.4)})--(\ref{eq:(3.5)}), we have $a\leftrightarrow b$. If $n=1$ it is trivially satisfied, if $n=2$, this was established in \cite[Thm 4.4]{DvLa1}. Assume that it holds for each $F: \mathbb R^i\to E$, satisfying (\ref{eq:(3.4)})--(\ref{eq:(3.5)}), where $i=1,\ldots,n-1$. If $(s_1,\ldots,s_n)$ and $(t_1,\ldots,t_n)$ are comparable, then clearly $a \leftrightarrow b$. Assume that $(s_1,\ldots,s_n)$ and $(t_1,\ldots,t_n)$ are not comparable. If there is $i=1,\ldots,n$ such that $s_i=t_i$, then for each $w_i \in \mathbb R$, $F_{w_i}(w_1,\ldots, w_{i-1},w_{i+1},
\ldots,w_n):=F(w_1,\ldots,w_n)$, $w_1,\ldots, w_{i-1},w_{i+1},\ldots,w_n \in \mathbb R$, is an $n-1$-dimensional mapping satisfying (\ref{eq:(3.4)})--(\ref{eq:(3.5)}), so by induction, $a$ and $b$ are compatible. If for each $i=1,\ldots,n$, we have $s_i\ne t_i$, we put $u_i=\min\{s_i,t_i\}$ and $v_i=\max\{s_i,t_i\}$ for each $i=1,\ldots,n$. Given $a_1=F(s_1,\ldots,s_n)- F(u_1,\ldots,u_n)$, $b_1=F(t_1,\ldots,t_n)-F(u_1,\ldots,u_n)$, and $c = F(u_1,\ldots,u_n)$, we have to show that $a_1+b=a_1+b_1+c$ exists in $E$.

We have $A:=(-\infty,v_1)\times \cdots \times (-\infty,v_n)=((-\infty,u_1)\cup \langle u_1,v_1))\times \cdots \times ((-\infty,u_n)\cup \langle u_n,v_n))$.
Applying Lemma \ref{le:4.3}, we see that for each element $C=C_1\times\cdots\times C_n$, where $C_i \in \{(-\infty,u_i),\langle u_i,v_i)\}$, we have $V_F(C)\in E$. This is true also for union of each finite system of mutually disjoint $C$'s,
which yields $a_1+ b_1+c$ is defined in $E$, consequently $F(s_1,\ldots,s_n)$ and $F(t_1,\ldots,t_n)$ are compatible.

Applying \cite[Thm 3.2]{RieZ} or \cite[Thm 1.10.20]{DvPu}, we see that the system $\{F(s_1,\ldots,s_n)\colon s_1,\ldots,s_n\in \mathbb R\}$ belongs to a block and this block is a $\sigma$-complete MV-sub-effect algebra of $E$, see \cite[Cor 4.4]{RieZ}. Consequently, $\{F(s_1,\ldots,s_n)\colon s_1,\ldots,s_n\in \mathbb R\}$ is in a $\sigma$-complete MV-algebra, applying Theorem \ref{th:4.1}, we can find an $n$-dimensional observable $x$ in question.
\end{proof}

Every effect-tribe $\mathcal T$ is in fact an interval in a Dedekind $\sigma$-complete po-group (not necessarily with interpolation) of bounded functions on $\Omega\ne \emptyset$, so that an $n$-dimensional spectral resolution can be defined as a mapping $F:\mathbb R^n \to \mathcal T$ such that (\ref{eq:(3.2)})--(\ref{eq:(3.5)}) hold.
If we take $n$-dimensional spectral resolutions on an effect-tribe, this is not covered by the above theorems, however, the one-to-one relationship between them and $n$-dimensional observables is relatively straightforward:

\begin{theorem}\label{th:4.5}
Let $\mathcal T$ be an effect-tribe of fuzzy sets on $\Omega \ne \emptyset$. Then for each $n$-dimensional spectral resolution $F:\mathbb R^n\to \mathcal T$, there is a unique $n$-dimensional observable $x$ on $\mathcal T$ such that $F(s_1,\ldots,s_n)=x((-\infty,s_1)\times\cdots\times (-\infty,s_n))$, $s_1,\ldots,s_n \in \mathbb R$.
\end{theorem}

\begin{proof}
If $F$ is an $n$-dimensional spectral resolution on the effect-tribe $\mathcal T$, then $F(s_1,\ldots,s_n)$ is a fuzzy set belonging to $\mathcal T$. Given $\omega \in \Omega$, the mapping $F_\omega: \mathbb R^n \to [0,1]$ given by $F_\omega(s_1,\ldots,s_n)=F(s_1,\ldots,s_n)(\omega)$, $s_1,\ldots,s_n \in \mathbb R$, is an $n$-dimensional distribution function on $\mathbb R^n$ and in view of \cite[Thm 2.25]{Kal}, there is a unique probability measure $P_\omega$ on $\mathcal B(\mathbb R^n)$ such that $P_\omega(s_1,\ldots,s_n)=P_\omega((-\infty,s_1)\times \cdots\times (-\infty,s_n))$, $s_1,\ldots,s_n \in \mathbb R$. If we define a mapping $x:\mathcal B(\mathbb R^n)\to \mathcal T$ by $x(A)(\omega)=P_\omega(A)$, $A \in \mathcal B(\mathbb R^n)$, then it is possible to show that $x$ is an $n$-dimensional observable such that $F(s_1,\ldots,s_n)=x((-\infty,s_1)\times\cdots\times (-\infty,s_n))$, $s_1,\ldots,s_n \in \mathbb R$. The uniqueness of $x$ follows from applications of the Sierpi\'nski Theorem, see \cite[Thm 1.1]{Kal}.
\end{proof}

The latter theorem can be applied for the case of effect algebra $\mathcal E(H)$ and $\mathcal P(H)$ which are important for the mathematical foundations of quantum mechanics. We note that $\mathcal E(H)$ is a monotone $\sigma$-complete effect algebra which is an interval in $\mathcal B(H)$, i.e. $\mathcal E(H)=\Gamma^e_a(\mathcal B(H),I)$, where $\mathcal B(H)$ is the system of all Hermitian operators on $H$ (it is a Dedekind $\sigma$-complete po-group without interpolation). The lattice properties of $\mathcal B(H)$, $\mathcal E(H)$ and $\mathcal P(H)$ are different because $\mathcal B(H)$ is due to Kadison's result, \cite[Thm 58.4]{LuZa}, an antilattice, $\mathcal E(H)$ is Dedekind monotone $\sigma$-complete, and $\mathcal P(H)$ is a complete lattice. We remind and an $n$-dimensional spectral resolution on $\mathcal E(H)$ ($\mathcal P(H)$) is a mapping $F:\mathbb R^n \to \mathcal E(H)$ ($F:\mathbb R^n \to \mathcal P(H)$) such that (\ref{eq:(3.1)})--(\ref{eq:(3.5)}) hold.

\begin{theorem}\label{th:4.6}
Let $F$ be an $n$-dimensional spectral resolution on $\mathcal E(H)$ and on $\mathcal P(H)$, respectively, where $H$ is a real, complex or quaternionic Hilbert space of any dimension. Then there is a unique $n$-dimensional observable $x$ on $\mathcal E(H)$ and on $\mathcal P(H)$, respectively, such that $F(s_1,\ldots,s_n)=x((-\infty,s_1)\times\cdots\times (-\infty,s_n))$, $s_1,\ldots,s_n \in \mathbb R$.
\end{theorem}

\begin{proof}
Let $\Omega(H)$ be the unit sphere in $H$, i.e. $\Omega(H)=\{\omega \in H \colon \|\omega\|=1\}$. If, given $A \in \mathcal R(H)$, we define $\mu_A(\omega)=(A\omega,\omega)$, $\omega \in \Omega(H)$, then $\mu_A$ is a fuzzy set on $\Omega(H)$, and $\mathcal T(H)=\{\mu_A\colon A \in \mathcal E(H)\}$ is an effect-tribe that is isomorphic to $\mathcal E(H)$ under the isomorphism $A \mapsto \mu_A$. If we apply Theorem \ref{th:4.5}, we can establish the one-to-one correspondence in question.

If $E =\mathcal P(H)$, we establish the result in the same way as for $\mathcal E(H)$.
\end{proof}

Finally, we say that a monotone $\sigma$-complete effect algebra satisfies the {\it Loomis--Sikorski Property} if there are an effect-tribe $\mathcal T$ of fuzzy sets on $\Omega\ne\emptyset$ and a surjective $\sigma$-homomorphism $\pi$ from $\mathcal T$ onto $E$ satisfying the lifting property. For example, every $\sigma$-complete MV-algebra and monotone $\sigma$-complete effect algebra with (RDP) have the Loomis--Sikorski Property (see the second proofs of Theorems \ref{th:4.1}--\ref{th:4.2}), as well as every effect-tribe.

\begin{theorem}\label{th:4.7}
Let $E$ be a monotone $\sigma$-complete effect algebra with the Loomis--Sikorski Property. Then there is a one-to-one correspondence between $n$-dimensional spectral resolutions and $n$-dimensional observables on $E$.
\end{theorem}

\begin{proof}
Let $F$ be an $n$-dimensional spectral resolution on $E$, $\mathcal T$ be an effect-tribe and let $\pi:\mathcal T\to E$ be a surjective $\sigma$-homomorphism with the lifting property. For the effect-tribe $\mathcal T$, we have $\mathcal T=\Gamma^e_a(\mathcal G,1_\Omega)$, where $(\mathcal G,1_\Omega)$ is a unital Dedekind $\sigma$-complete po-group of all bounded functions on $\Omega$ with the pointwise ordering. According to Theorem \ref{th:lifting} (which holds also for this  case), there is an almost $n$-dimensional spectral resolution $K$ on $\mathcal T$ such that $\pi\circ K=F$. The final conclusion follows easily from Theorem \ref{th:4.5} when we apply it to $K$.
\end{proof}

Finally, we note that it would be interesting to extend the one-to-one relationship between $n$-dimensional spectral resolutions and $n$-dimensional observables also for other classes of MV-algebras and effect algebras.

\section{Joint $n$-Dimensional Observables}

In the section, we apply the one-to-one relationship between $n$-dimensional spectral resolutions and $n$-dimensional observables to define three different kinds of joint $n$-dimensional observables of $n$ one-dimensional observables defined on $\sigma$-complete MV-algebras.

Now, we show that given $n$ one-dimensional observables $x_1,
\ldots, x_n$ on a $\sigma$-complete MV-algebra $M$, there is a joint $n$-dimensional observable $x$ on $M$ such that
\begin{equation}\label{eq:joint}
x((-\infty,s_1)\times\cdots\times (-\infty,s_n))=\bigwedge_{i=1}^n x_i((-\infty,s_i)),\quad s_1,\ldots,s_n \in \mathbb R.
\end{equation}

We remind the following known result.

\begin{lemma}\label{le:4.7}
Let $\{x_i\colon i \in I\}$ be a system of elements of an MV-algebra $M$.

{\rm (1)} Let $\bigvee_{i\in I} x_i$ exist in $M$, and let $x$ be any element of $M$. Then $\bigvee_{i\in I}(x\wedge x_i)$ exists in $M$ and
\begin{equation}\label{eq:2.1}
\bigvee_{i\in I}(x\wedge x_i)= x\wedge \bigvee_{i\in I}x_i.
\end{equation}

{\rm(2)} If $\bigwedge_{i\in I} x_i$ exists in $M$, then for each $x\in M$, the element $\bigwedge_{i\in I}(x\vee x_i)$ exists in $M$ and
\begin{equation}\label{eq:2.2}
\bigwedge_{i\in I}(x\vee x_i)= x\vee \bigwedge_{i\in I} x_i.
\end{equation}
\end{lemma}

\begin{theorem}\label{th:4.8}
Let $x_1,\ldots,x_n$ be one-dimensional spectral resolutions on a $\sigma$-complete MV-algebra $M$. Then there is a unique $n$-dimensional observable $x$ on $M$ such that {\rm (\ref{eq:joint})} holds.
\end{theorem}

\begin{proof}
Let $F_i(s)=x_i((-\infty,s))$, $s \in \mathbb R$, be a one-dimensional spectral resolution corresponding to $x_i$.
The mapping $F: \mathbb R^n \to M$ defined by $F(s_1,\ldots,s_n)=\bigwedge_{i=1}^n F_i(s_i)$, $s_1,\ldots,s_n\in \mathbb R$, satisfies conditions (\ref{eq:(3.2)})--(\ref{eq:(3.4)}), see e.g. Lemma \ref{le:4.7}. We assert that $F$ is an $n$-dimensional spectral resolution. To show that, we have to prove that $F$ satisfies the volume condition (\ref{eq:(3.5)}).

The volume condition will be established if it will hold for every linearly ordered MV-algebra. Hence, we present the following claim:

\vspace{3mm}
{\it Claim. Let $F_1,\ldots,F_n$, $n\ge 2$, be functions from $\mathbb R$ into a linearly ordered MV-algebra $M$ such that each $F_i$ satisfies the volume condition.  Then $F(s_1,\ldots,s_n)=\bigwedge_{i=1}^n F_i(s_i)$, $s_1,\ldots,s_n\in \mathbb R$, satisfies the volume condition. If, given $A=\langle a_1,b_1)\times\cdots\times \langle a_n,b_n)$, we can assume that $F_1(a_1)\le F_2(a_2)\le \cdots\le F_n(a_n)$, then}
\begin{equation}\label{eq:Delta}
\Delta_1(a_1,b_1)\cdots\Delta_n(a_n,b_n)F(s_1,\ldots,s_n)= \bigwedge_{i=1}^{n-1}(F_i(b_i) \wedge F_n(b_n))- \bigwedge_{i=1}^{n-1}(F_i(b_i)\wedge F_n(a_n)).
\end{equation}

\begin{proof}
Let $A=\langle a_1,b_1)\times \cdots\times \langle a_n,b_n)$ be a semi-closed $n$-dimensional cube in $\mathbb R^n$, where $a_i\le b_i$ are real numbers for each $i=1,\ldots,n$. For simplicity, we define $x^i_0=F_i(a_i)$ and $x^i_1=F_i(b_i)$ for $i=1,\ldots,n$.

We note that a mapping $f:\mathbb R\to M$ satisfies the volume condition iff $f$ is non-decreasing. Therefore, the mapping $\widehat F_i:\mathbb R\to M$ defined by $\widehat F_i(x)=x^1_1\wedge F_i(x)$, $x \in \mathbb R$, is non-decreasing, so that it satisfies the volume condition for each $i=1,\ldots,n$. We set by $\widehat F(s_1,\ldots,s_n)=\bigwedge_{i=1}^n \widehat F_i(s_i)$, $s_1,\ldots,s_n\in \mathbb R$. We denote by $A_1=\langle a_2,b_2)\times\cdots\times \langle a_n,b_n)$.

If $V_F(A)$ is the expression on the left-hand side of (\ref{eq:Delta}), then
\begin{equation}\label{eq:Vol}
V_F(A)= \sum_{\phi\in \{0,1\}^{\{1,\ldots,n\}}} \mathrm{sgn}(\phi)\cdot\bigwedge_{i=1}^n x_{\phi(i)}^i,
\end{equation}
where $\mathrm{sgn}(\phi)$ equals $+1$ iff $|\phi^{-1}(0)|$ is even and equals $-1$ otherwise (we use that $M$ is an interval in some linearly ordered group $G$ with $0\cdot a:=0$ and $-1\cdot a:=-a$ for each $a \in M$).

To prove the volume condition, we use an induction on $n$. In the case $n=1$, the formula~\eqref{eq:Vol} has a form $x_1^1- x_0^1\ge 0$, which clearly holds. Now suppose the case $n>1$. As we have said, without loss of generality, we can assume $x^1_0\le x^2_0\le \cdots\le x^n_0$.
The expression for~\eqref{eq:Vol} equals
\begin{align*}
V_F(A)&=\sum_{\phi\in \{0,1\}^{\{2,\ldots,n\}}} -\mathrm{sgn}(\phi)\cdot \big(x^1_0\wedge\bigwedge_{i=2}^n x_{\phi(i)}^i\big)+\sum_{\phi\in \{0,1\}^{\{2,\ldots,n\}}} \mathrm{sgn}(\phi)\cdot \big(x^1_1\wedge\bigwedge_{i=2}^n x_{\phi(i)}^i\big)\\
&=\sum_{\phi\in \{0,1\}^{\{2,\ldots,n\}}} \mathrm{sgn}(\phi)\cdot \big( x^1_1\wedge\bigwedge_{i=2}^n x_{\phi(i)}^i\big)= \sum_{\phi\in \{0,1\}^{\{2,\ldots,n\}}} \mathrm{sgn}(\phi)\cdot \bigwedge_{i=2}^n(x^1_1\wedge x_{\phi(i)}^i)\\
&= V_{\widehat F}(A_1)\ge 0.
\end{align*}
The first summand vanishes as it equals $-\sum_\phi \mathrm{sgn}(\phi)\cdot x_0^1$, where $\phi$ goes through all the functions in $\{0,1\}^{\{2,\ldots,n\}}$, we see that exactly half of the functions has the negative sign. The second summand satisfies~ \eqref{eq:Delta}--\eqref{eq:Vol} by the induction hypothesis, so also \eqref{eq:Delta} is satisfied, so that Claim is established.
\end{proof}

As a final conclusion, $F$ satisfies the volume condition on $M$, and therefore, $F$ is an $n$-dimensional spectral resolution on $M$ for each $\sigma$-complete MV-algebra $M$. Applying Theorem \ref{th:4.1}, there is a unique $n$-dimensional observable $x$ on $M$ satisfying \eqref{eq:joint}.
\end{proof}

The $n$-dimensional observable $x$ from the latter theorem is said to be an {\it $n$-dimensional meet joint observable} of $x_1,\ldots, x_n$. We note that using the Sierpi\'nski Theorem, we can show

\begin{equation}\label{eq:joint1}
x(\pi_i^{-1}(A))=x_i(A),\quad A \in \mathcal B(\mathbb R),\ i=1,\ldots,n,
\end{equation}
where $\pi_i:\mathbb R^n \to \mathbb R$ is the $i$-th projection.

In addition, from \eqref{eq:joint1}, we can prove

\begin{equation}\label{eq:joint2}
x(A_1\times\cdots\times A_n)\le \bigwedge_{i=1}^n x_i(A_i),\quad A_1,\ldots, A_n \in \mathcal B(\mathbb R),
\end{equation}
and in general, it can happen that in \eqref{eq:joint2} we have strict inequality.

Now, we define a second type of joint $n$-dimensional observables on MV-algebras with product.

We say that an MV-algebra $M$ admits a {\it product}, see \cite{DiDv}, if there is a commutative and associative binary operation $\cdot$ on $M$ satisfying for all
$a,b,c \in M$
\begin{itemize}
\item[(i)] if $a + b$ is defined in $M$, then $(a\cdot c) + (b \cdot c)$ and
$(c \cdot a) + (c \cdot b)$ exist and

\begin{align*}
(a+b)\cdot c &= (a\cdot c) + (b \cdot c),\\
c\cdot (a + b) &= (c \cdot a) + (c \cdot b),
\end{align*}

\item[(ii)] $a\cdot 1 = a$,
\end{itemize}
and we say that $M$ is a {\it product MV-algebra}. For example, the MV-algebra of the real interval $M=\Gamma(\mathbb R,1)$ admits a product which is the standard product of reals. Basic properties of product MV-algebras are (a)
$a \cdot 0 = 0 = 0 \cdot a$, (b) if $a\le b$, then for any $c \in M$, $a\cdot c \le b\cdot c$ and $c \cdot a \le c \cdot b$. We note that if $M$ is $\sigma$-complete, then it is easy to show that $\{a_i\}_i\nearrow a$ imply $\{b\cdot a_i\}_i\nearrow b\cdot a$ for each $b \in M$, see \cite[Thm 5.8]{DvLa1}.

\begin{theorem}\label{th:4.9}
Let $x_1,\ldots,x_n$ be one-dimensional observables, $n\ge 1$, on a $\sigma$-complete MV-algebra $M$ with a product, and let $F_1,\ldots,F_n$ be the corresponding one-dimensional spectral resolutions. If we define
$$
F(s_1,\ldots,s_n)=\prod_{i=1}^nF_i(s_i),\quad s_1,\ldots,s_n\in \mathbb R,
$$
then $F$ is an $n$-dimensional observable on $M$ and there is a unique $n$-dimensional observable $x$, such that
\begin{equation}\label{eq:6.3}
x(A_1\times\cdots \times A_n)=\prod_{i=1}^n x_i(A_i),\quad A_1,\ldots,A_n\in \mathcal B(\mathbb R).
\end{equation}
\end{theorem}

\begin{proof}
The mapping $F$ satisfies (\ref{eq:(3.1)})--(\ref{eq:(3.4)}). To show the volume condition, let a semi-closed rectangle  $\langle a_1,b_1)\times \cdots\times \langle a_n,b_n)$ be given. If we denote by $V^{(b_1,\ldots,b_n)}_{(a_1,\ldots,a_n)}(F)$ the left-hand side of volume condition (\ref{eq:(3.5)}), then we have
$$ V^{(b_1,\ldots,b_n)}_{(a_1,\ldots,a_n)}(F)=\prod_{i=1}^n (F_i(b_i)-F_i(a_i))\ge 0.$$
Applying Theorem \ref{th:4.1}, we see that there is a unique $n$-dimensional observable of $x_1,\ldots,x_n$ determined by the $n$-dimensional spectral resolution $F$.

Using mathematical induction and applying the Sierpi\'nski Theorem, it is possible to establish (\ref{eq:6.3}).
\end{proof}

The $n$-dimensional observable $x$ from Theorem \ref{th:4.9} is said to be an {\it $n$-dimensional product joint observable} of $x_1,\ldots, x_n$. Clearly, we have $$
\prod_{i=1}^nF_i(t_i)\le \bigwedge_{i=1}^n F_i(t_i),\quad t_1,\ldots,t_n\in \mathbb R,
$$
however, it can happen that the $n$-dimensional meet joint observable is different of the $n$-dimensional product joint observable of one-dimensional observables $x_1,\ldots,x_n$.

Now, we formulate three open questions concerning also a possible third kind of an $n$-dimensional joint observable.

\subsection{$n$-dimensional Spectral $\odot$-joint Observable}
Let $x_1,\ldots,x_n$ be one-dimensional observables, $n\ge 1$, on a $\sigma$-complete MV-algebra $M$, and let $F_1,\ldots,F_n$ be the corresponding one-dimensional spectral resolutions. We define
\begin{equation}\label{eq:6.4}
F^\odot(s_1,\ldots,s_n)=F_1(s_1)\odot \cdots \odot F_n(s_n),\quad s_1,\ldots,s_n \in \mathbb R.
\end{equation}
It is clear that $F^\odot$ satisfies (\ref{eq:(3.1)})--(\ref{eq:(3.4)}).
If $n=2$, in \cite{DvLa2}, there was shown that $F^\odot$ is a two-dimensional spectral resolution.

Show that this is true for each $n\ge 3$.

In the positive answer, $F^\odot$ can be extended to a unique $n$-dimensional observable $x^\odot$ such that
$$F^\odot(s_1,\ldots,s_n)=x^\odot((-\infty,s_1,)\times \cdots\times (-\infty,s_n)),\quad s_1,\ldots,s_n \in \mathbb R.
$$
The observable $x^\odot$ is said to be an $n$-{\it dimensional $\odot$-joint observable} of $x_1,\ldots,x_n$.

\subsection{Group Joint Observables}

Let $x_1,\ldots,x_k$ be $n_1-,\ldots,n_k-$dimensional observables on a $\sigma$-complete MV-algebra $M$ such that $n=n_1+\cdots+n_k$. Let $F_i$ be the $n_i$-dimensional spectral resolutions corresponding to $x_i$, $i=1,\ldots,k$.
Define for all $s_1,\ldots,s_n \in \mathbb R$ the mappings
\begin{equation}\label{eq:6.5}
F(s_1,\ldots,s_n)=F_1(s_1,\ldots,s_{n_1})\wedge F_2(s_{n_1+1},\ldots s_{n_1+n_2})\wedge \cdots\wedge F_k(s_{n_1+\cdots+n_{k-1}+1},\ldots,s_n)
\end{equation}
and
\begin{equation}\label{eq:6.6}
F^\odot(s_1,\ldots,s_n)=F_1(s_1,\ldots,s_{n_1})\odot F_2(s_{n_1+1},\ldots s_{n_1+n_2})\odot \cdots\odot F_k(s_{n_1+\cdots+n_{k-1}+1},\ldots,s_n).
\end{equation}
Exhibit whether $F$ and $F^\odot$ are $n$-dimensional spectral resolutions.

We note that it can happen that $F$ does not satisfy the volume condition. Indeed, let $(G,u)=(\mathbb R,1)$ and define a one-dimensional observable $x(\{0\})=0.05$, $x(\{1\})=0.3$, $x(\{2\})=0.65$ and a two-dimensional observable $y(\{0,0\})=0.1$, $y(\{0,1\})=0.1$, $y(\{1,0\})=0.2$, $y(\{2,2\})=0.6$. Then for $A=\langle 0.1,1.1)\times \langle 0.1,1.1)\times \langle 0.1,1.1)$ we have $f_0=0.05$, $f_1=0.35$ and $g_{00}=0.1$, $g_{01}=0.2$, $g_{10}=0.3$, $g_{11}=1.0$. Then $(f_1\wedge g_{11})+ (f_1\wedge g_{00})+ (f_0\wedge g_{10})+ (f_0\wedge g_{01})= 0.55< 0.6= (f_1\wedge g_{10})+(f_1\wedge g_{01})+(f_0\wedge g_{11})+(f_0\wedge g_{00})$.

On the other hand, if in (\ref{eq:6.5}), every $F_i$ is of the form (\ref{eq:joint}), then $F$ is an $n$-dimensional spectral resolution.

\subsection{$n$-dimensional Observables and Joint Distribution}
If $s$ is a $\sigma$-additive state on a $\sigma$-complete MV-algebra $M$ and $x$ is an $n$-dimensional observable, then the mapping $s_x: \mathcal B(\mathbb R^n)\to[0,1]$ defined by
$$
s_x(A)=s(x(A)),\quad A \in \mathcal B(\mathbb R^n),
$$
is a probability measure on $\mathcal B(\mathbb R^n)$. Therefore, if $x_1,\ldots,x_n$ are one-dimensional observables and if $x$ is their $n$-dimensional meet joint observables, then we have
$$
s_x(E_1\times\cdots\times E_n)=s(\bigwedge_{i=1}^n x_i(E_i)),\quad E_1,\ldots, E_n \in \mathcal B(\mathbb R),
$$
and $s_x$ is a joint distribution of $x_1,\ldots,x_n$ in the state $s$. Therefore, this can be useful for a multivariate analysis. For example, if $x$ is a one-dimensional observable and $f:\mathbb R\to \mathbb R$ is a Borel measurable function, then $f(x):\mathcal B(\mathbb R)\to M$ defined by $f(x)(E)=x(f^{-1}(E))$, $E \in \mathbb R$, is also an observable on $M$.  For example if $f(t)=t^k$, we write $f(x)=x^k$. Then the $k$-th moment of $x$ in the state $s$ is
$$ Exp(x^k)=\int_{\mathbb R} t^k \dx s_x(t)
$$
if the integral exists and is finite.

If $x$ is a two-dimensional meet joint observable of $x_1$ and $x_2$, we can define for example
$$
\mu_{11}=\int_{\mathbb R}\int_{\mathbb R} uv\dx s_x(u,v)
$$
if the double integral exists and is finite, etc.

Analogous ideas can be done also for an $n$-dimensional product joint observable of $x_1,\ldots,x_n$.

\section{Conclusion}

Any measurement of $n$ observables in quantum structures is modeled by an $n$-dimensional observable which is a kind of a $\sigma$-homomorphism from the Borel $\sigma$-algebra $\mathcal B(\mathbb R^n)$ into the quantum structure which is a monotone $\sigma$-complete effect algebra or a $\sigma$-complete MV-algebra. Every observable $x$ restricted to infinite intervals of the form $(-\infty,t_1)\times\cdots\times (-\infty,t_n)$, $t_1,\ldots,t_n\in \mathbb R$, defines an $n$-dimensional spectral resolution which is characterized as a mapping from $\mathbb R^n$ into the quantum structure that is monotone, with non-negative increments, and is going to $0$ if one variable goes to $-\infty$ and going to $1$ if all variables go to $+\infty$.

Our main task is to find conditions when given an $n$-dimensional spectral resolution can be extended to a (unique) $n$-dimensional observable. To show that, we studied a possibility of lifting an $n$-dimensional spectral resolution from one effect algebra to the second one.
In the paper we have exhibit the following situation: Let a monotone $\sigma$-complete effect algebra $E=\Gamma^e_a(H,v)$ be a $\sigma$-homomorphic image of another monotone $\sigma$-complete effect algebra $\Gamma^e_a(G,u)$ under homomorphism $\pi$. We show that given an $n$-dimensional spectral resolution $F$ on $E$ can be lifted to an almost $n$-dimensional spectral resolution $K$ on $\Gamma^e_a(G,u)$ such that $\pi\circ K=F$, see Theorem \ref{th:lifting}.

Having this lifting, we are able to show that every $n$-dimensional spectral resolution on a $\sigma$-complete MV-algebra, Theorem \ref{th:4.1}, or on a monotone $\sigma$-complete effect algebra with (RDP), Theorem \ref{th:4.2}, can be extended to a unique $n$-dimensional observable. In addition, we present two kinds of proofs of these extensions, one using the Loomis--Sikorski representation theorem, the second one is using the mentioned lifting. Moreover, we extended the one-to-one correspondence between $n$-dimensional spectral resolutions and $n$-dimensional observables also for $\sigma$-complete effect algebras, Theorem \ref{th:4.4}, effect-tribes, Theorem \ref{th:4.5}, and monotone $\sigma$-complete effect algebras with the Loomis--Sikorski Property, Theorem \ref{th:4.7}.

Finally, we apply the extension of $n$-spectral resolution to define three different kinds of joint $n$-dimensional observable of $n$ one-dimensional observables on $\sigma$-complete MV-algebras, see Theorem \ref{th:4.8} and Section 6.

\end{document}